\documentclass[12pt,a4paper]{amsart}
\calclayout
\numberwithin{equation}{section}
\allowdisplaybreaks
\theoremstyle{plain}

\newtheorem{theorem}{Theorem}[section]
\newtheorem{lemma}[theorem]{Lemma}

\newtheorem{corollary}[theorem]{Corollary}
\newtheorem{remark}{Remark}

\numberwithin{equation}{section}

\usepackage{enumerate}

\usepackage{amssymb}
\usepackage{mathtools}
\mathtoolsset{showonlyrefs}
\usepackage{mathrsfs}
\usepackage{comment}
\usepackage{hyperref}
\usepackage{xcolor}
\usepackage{footnote}

\def\P{\mathbb{P} }
\def\R{\mathbb{R} }

\def\E{\mathbb{E} }

\def\1{\mathbf{1}}

\begin{document}
	\title[Subcritical branching killed Brownian motion]{
		Asymptotic behaviors of subcritical branching killed Brownian motion with drift}
\thanks{The research of this project is supported by the National Key R\&D Program of China (No. 2020YFA0712900).}
\thanks{The research of Yan-Xia Ren is supported by NSFC (Grant Nos. 12071011 and 12231002) and
the Fundamental Research Funds for the Central Universities, Peking University LMEQF}
\thanks{ Research supported in part by a grant from the Simons	Foundation	(\#960480, Renming Song).}
\thanks{*Yaping Zhu is the cooresponding author}
	\author[H. Hou]{Haojie Hou}
\author[Y.-X. Ren]{Yan-Xia Ren}
\author[R. Song]{Renming Song}
\author[Y. Zhu]{Yaping Zhu$^\ast$}
	\address{Haojie Hou\\ School of Mathematics and Statistics \\ Beijing Institute of Technology \\ Beijing 100081\\ P. R. China}
%HJ	\email{TBD}
   \email{houhaojie@bit.edu.cn}
	
	\address{Yan-Xia Ren\\ LMAM School of Mathematical Sciences \& Center for
		Statistical Science  \\ Peking University\\ Beijing 100871\\ P. R. China}
	\email{yxren@math.pku.edu.cn}
	
	\address{Renming Song\\ Department of Mathematics \\ University of Illinois Urbana-Champaign \\ Urbana, IL 61801 \\ U.S.A.}
	\email{rsong@illinois.edu}
	
	\address{Yaping Zhu\\ School of Mathematical Sciences \\ Peking University\\ Beijing 100871\\ P. R. China}
	\email{zhuyp@pku.edu.cn}

	\begin{abstract}
	In this paper, we study asymptotic behaviors of a subcritical branching killed Brownian motion with drift $-\rho$ and offspring distribution $\{p_k:k\ge 0\}$.  Let $\widetilde{\zeta}^{-\rho}$ be the extinction time of this subcritical branching killed Brownian motion, $\widetilde{M}_t^{-\rho}$ the maximal position of all the particles alive at time $t$ and $\widetilde{M}^{-\rho}:=\max_{t\ge 0}\widetilde{M}_t^{-\rho}$ the all time maximal position. Let $\P_x$ be the law of this subcritical branching killed Brownian motion when the initial particle is located at $x\in (0,\infty)$.
	Under the assumption $\sum_{k=1}^\infty k  (\log k) p_k <\infty$, we establish the decay rates of  $\P_x(\widetilde{\zeta}^{-\rho}>t)$ and $\P_x(\widetilde{M}^{-\rho}>y)$ as $t$ and $y$ tend to $\infty$ respectively. We also establish the decay rate of  $\P_x(\widetilde{M}_t^{-\rho}>z(t,\rho))$ as $t\to\infty$,
	where $z(t,\rho)=\sqrt{t}z-\rho t$ for $\rho\leq 0$ and $z(t,\rho)=z$ for $\rho>0$.
	As a consequence, we obtain a Yaglom-type limit theorem.
	\end{abstract}

	\maketitle

	\section{Introduction}\label{Sec1}

   A branching Brownian motion (BBM) with drift $-\rho$ is a continuous-time Markov process defined as follows: at time 0, there is a particle at $x\in \R$ and this particle moves according to a  Brownian  motion with drift $-\rho\in \R$.  After an exponential time with parameter $\beta>0$, independent of the spatial motion, this particle dies and is replaced by $k$ offspring with probability $p_k$, $k\ge 0$. The offspring move independently according to Brownian  motion with drift $-\rho$ from the place where they are born and obey the same branching mechanism as their parent. This procedure goes on.	Let $N_t^{-\rho}$ be the collection of particles alive at time $t$.  If $u \in N_t^{-\rho}$, let $X_u(t)$ denote the position of the particle $u$ at time $t$ and for $s\in (0,t)$, we denote by $X_u(s)$ the position at time $s$ of the ancestor of $u$.
 The point process $(Z_t^{-\rho})_{t\ge 0}$ defined by
	$$
	Z_t^{-\rho}:=\sum_{u\in N_t^{-\rho}}\delta_{X_u(t)}, \quad t\ge 0,
	$$
	is called a branching Brownian motion with drift $-\rho$. We will use $\mathbb{P}_x$ to denote the law of this process and use $\mathbb{E}_x$ to denote the corresponding expectation. 	Let
	\[
	\zeta:= \inf\{t>0, N_t^{-\rho}= \emptyset\}
	\]
	be the extinction time of $(Z_t^{-\rho})_{t\ge 0}$. Note that the law of $\zeta$ does not depend on $\rho$ and is equal to that of the extinction time of the continuous-time Galton-Waston process with the same branching mechanism as the BBM. Let $m:= \sum_{k=0}^\infty kp_k $ be the mean number of offspring and let $f$ be the generating function of the offspring distribution, i.e. $f(u)=\sum_{k=0}^{\infty}p_ku^k,u\in[0,1]$. It is well-known that the process
	will become extinct in finite time with probability $1$ if and only if $m<1$ (subcritical) or $m=1$ and $p_1\neq 1$ (critical).  When $m>1$ (supercritical), the process survives with positive probability.
	
	For any $t\ge0$, let
	 \begin{align}
		M_t^{-\rho}:=\max\{X_u(t):u\in N_t^{-\rho}\}
	\end{align}
	be the maximal position of all the particle alive at time $t$ and let
	\[
		M^{-\rho} := \sup_{t>0} M_t^{-\rho}
	\]	
	be the all time maximal position. In the subcritical and critical cases, $\P_x\left( M^{-\rho} <\infty \right)=1$
	 for any $x, \rho\in \R$.
	
	In the critical case
	 $m=1$ and $p_1\neq 1$, Sawyer and Fleischman \cite{Sawyer79}  proved
	that if	$\beta=1$ and	the offspring distribution has 	finite third moment, then
	\begin{align}\label{tail-prob-M}
		\lim_{x\to\infty}x^2 \mathbb{P}_0(M^0\ge x)
		=\frac{6}{\sigma^2 },
	\end{align}
	where $\sigma$ is the variance of the offspring distribution.
	For a critical branching random walk with spatial motion having finite
	$(4+\varepsilon)$th
	moment, a similar result as \eqref{tail-prob-M} was proved by Lalley and Shao \cite{Lalley15}. It was also proved in \cite{Lalley15}  that the law of	$M_t^0/\sqrt{t}$ under $\P_0\left(\cdot | \zeta>t\right)$ converges weakly to some random variable.
	For related results	in the case of	 critical branching L\'{e}vy  processes, see \cite{Profeta24}.

	In the subcritical case $m\in (0,1)$, let
	\begin{align}\label{Def-alpha}
		\alpha:= \beta (1-m)\in (0,\infty).
	\end{align}
	Define
	\begin{align}\label{def_Phi}
		\Phi(u):= \beta\left(f(1-u)-(1-u)\right)=: \left(\alpha +\varphi(u)\right)u, \quad u\in [0, 1],
	\end{align}
	where $\varphi(u)=\frac{\Phi(u)-\alpha u}{u}$ for $u\in (0, 1]$ and $\varphi(0)= \Phi'(0+)-\alpha=0$.
	It is well-known (see Theorem 2.4 in \cite[p.121]{AH1983})	that the limit
	\begin{align}\label{Decay-Survival-Probability}
		\lim_{t\to\infty} e^{\alpha t} \P_0(\zeta>t)= C_{sub}\in (0,\infty)
	\end{align}
	if and only if
	\begin{align}\label{LLogL-moment-condition}
		\sum_{k=1}^\infty   k (\log k )  p_k <\infty.
	\end{align}
Now we give another equivalent form of \eqref{Decay-Survival-Probability}.
 For any $t>0$, define
	\begin{align}\label{def_g}
		g(t):=\mathbb{P}_0(\zeta>t).
	\end{align}
	It is well-known that $g(t)$ satisfies the equation
	\[
	\frac{\mathrm{d}}{\mathrm{d} t} g(t)= -\Phi(g(t))= -\left(\alpha + \varphi (g(t))\right)g(t),
	\]
	thus
	\begin{align}\label{e:rsnew}
		e^{\alpha t}g(t) = \exp\left\{-\int_0^t \varphi(g(s))\mathrm{d}s \right\}.
	\end{align}
    It follows from  \eqref{Decay-Survival-Probability} that
	\begin{align}\label{Constant-C-sub}
		C_{sub}= \exp\left\{-\int_0^\infty \varphi(g(s))\mathrm{d}s \right\}.
	\end{align}
	Therefore, \eqref{Decay-Survival-Probability} is equivalent to
	\begin{align}\label{Integral-of-phi-is-finite}
		\int_0^\infty \varphi(g(s))\mathrm{d} s<\infty.
	\end{align}

	For $M^{-\rho}$,
	when the underlying motion is a standard Brownian motion and the offspring distribution has finite third moment,  it was proved in \cite{Sawyer79} that,
 if $\rho=0$,
	\begin{align}\label{tail-prob-sub-M}
		\lim_{x\to\infty}\frac{\mathbb{P}_0
			(M^0>x)
			}{(1-m)s(x)e^{-\sqrt{2\alpha}x}}=1,
	\end{align}
	where $s(x)$ is a bounded positive function. The limit \eqref{tail-prob-sub-M} was later generalized in \cite{Profeta24} to subcritical branching spectrally negative L\'{e}vy processes.
	When specialized to our setting, \cite[Theorem 1.1]{Profeta24}  says
	that when $\sum_{k=0}^\infty k^3 p_k <\infty$,  there exists a constant $\kappa \in (0,\infty)$ such that
	\begin{align}\label{Generalized-Tail-of-M}
		\lim_{x\to\infty} e^{\left(\rho + \sqrt{2\alpha+\rho^2}\right) x}\P_0\left(M^{-\rho} \geq x\right) =\kappa.
	\end{align}
	In the case of subcritical branching random walks, it was proved in \cite[Theorem 1.2]{Neuman17}
	that when the random walk has finite range and is nearly right-continuous in the sense of \cite{Neuman17}, a 	similar result as \eqref{tail-prob-sub-M} holds. In \cite{Neuman17}, the authors
	also gave some estimates for the limit behavior of $\P_0(M^0\geq x)$ in the case of
	 general subcritical branching random walks.	For related results about near-critical branching random walks, see \cite{Neuman21}.

	In this paper, we are interested in the asymptotic behaviors of  branching killed Brownian motions with drift $-\rho$, in which particles are killed (along with their   descendants) upon hitting the origin. The point process
$(\widetilde{Z}_t^{-\rho})_{t\ge 0}$ defined by
	\begin{align}
		\widetilde{Z}_t^{-\rho}:=\sum_{u\in N_t^{-\rho}}1_{\{\min_{s\le t}X_u(s)>0\}}\delta_{X_u(t)},
	\end{align}
	is called a branching killed Brownian motions with drift $-\rho$.  Let
	\[
	\widetilde{\zeta}^{-\rho}:=\inf\left\{t\ge 0:\widetilde{Z}_t^{-\rho}((0,\infty))=0\right\}
	\]
	be the extinction time of $(\widetilde{Z}_t^{-\rho})_{t\ge 0}$.  We define the maximal position at
	time $t$  and the all time maximal position of $(\widetilde{Z}_t^{-\rho})_{t\ge 0}$ by
	\begin{align}\label{Def-Maximal-displacement}
		\widetilde{M}_t^{-\rho}:=\max_{u\in N_t^{-\rho}: \min_{s\leq t} X_u(s)>0}X_u(t)\quad\mbox{and}\quad \widetilde{M}^{-\rho}:=\max_{t\ge 0}\widetilde{M}_t^{-\rho}.
	\end{align}
	In the critical case ($m=1$ and $p_1\neq 1$),  Lalley and Zheng \cite[Theorem 6.1]{LZ15} proved that, if $\sum_{k=0}^{\infty}k^3p_k<\infty$, then
	\begin{align}
		\lim_{y\to\infty}y^3\P_x(\widetilde{M}^0 \ge y)=C_1 x,
	\end{align}
	where $C_1\in (0,\infty)$ is  a constant independent of $x$. It was also shown in \cite[Theorem 6.1]{LZ15} that,  for any $s\in (0,1)$,
	\begin{align}
		\lim_{y\to\infty}y^2\P_{sy}(\widetilde{M}^0\ge y)=C_2(s)\in (0, \infty).
	\end{align}
	Recently, Hou et al. \cite{Hou24} studied the asymptotic behaviors of the tails of
	the extinction time and the maximal  displacement of critical branching killed L\'{e}vy processes under some assumptions on the spatial motion and the assumption that the offspring distribution belongs to the domain of attraction of an $\alpha$-stable distribution for $\alpha \in(1,2]$.

	There are also quite a few papers in the literature studying the asymptotic behaviors of supercritical (i.e., $m\in (1, \infty)$)  branching killed
	Brownian motions  with drift $-\rho$.
	Kesten \cite{Kesten78} proved that, when  $\rho> \sqrt{2\beta (m-1)}$,  the process will become extinct almost surely and Harris and Harris \cite[Theorem 1]{Harris07} obtained the asymptotic behavior of the survival probability.  In the case $\rho <\sqrt{2\beta (m-1)}$, 	Harris, Harris and Kyprianou \cite{Harris06}  investigated the large deviation probability of	maximal position.
	 For related results  in the critical case $\rho =\sqrt{2\beta (m-1)}$, see \cite{Berestycki14, Kesten78,Maillard23,Maillard22}.

	The main focus of this paper is on the asymptotic behaviors of subcritical branching killed Brownian motions with drift.
	More precisely, we will study the asymptotic behaviors of $\P_x\left(\widetilde{\zeta}^{-\rho}>t\right)$ and $\mathbb{P}_x(\widetilde{M}^{-\rho}> y)$ as $t$ and
$y$ tend to $\infty$, respectively.
We will also study the decay rate of  $\mathbb{P}_x(\widetilde{M}_t^{-\rho}>z(t,\rho))$, where $z(t,\rho)= \sqrt{t}z-\rho t$ for $\rho\leq 0$ and $z(t,\rho)= z$ for $\rho>0$.

	Our first main result is as follows. Recall that $C_{sub}$ is given in \eqref{Decay-Survival-Probability}.  Also, the notation $f(t)\sim g(t)$ as $t\to a$ means that $\lim_{t\to a}f(t)/g(t)=1$.

	\begin{theorem}\label{thm1}
		Suppose that \eqref{LLogL-moment-condition} holds and $x>0$.
		\begin{enumerate}
			\item[(i)] If $\rho =0$, then
			\begin{align}
			\lim_{t\to\infty}\sqrt{t}e^{\alpha t}\P_x\left(\widetilde{\zeta}^{-\rho}>t\right)= \sqrt{\frac{2}{\pi}}C_{sub} x.
			\end{align}
			\item[(ii)] If $\rho<0$, then
			\begin{align}
				\lim_{t\to\infty}e^{\alpha t}\P_x\left(\widetilde{\zeta}^{-\rho}>t\right)= C_{sub}(1-e^{2\rho x}).
			\end{align}
			\item[(iii)] If $\rho>0$, then
			\begin{align}
				\lim_{t\to\infty}t^{\frac{3}{2}}e^{(\alpha+\frac{\rho^2}{2}) t}\P_x\left(\widetilde{\zeta}^{-\rho}>t\right)= \sqrt{\frac{2}{\pi}}C_0(\rho) xe^{\rho x},
			\end{align}
			where $C_0(\rho):=\lim_{N\to\infty}e^{(\alpha+ \frac{\rho^2}{2})N}
			\int_{0}^{\infty}ye^{-\rho y} \P_y\left(\widetilde{\zeta}^{-\rho}>N\right) \mathrm{d}y \in (0,\infty)$.
		\end{enumerate}
	   Furthermore, for any $\rho\in \R$, as $t\to\infty$,
	    \[
	    \P_x\left( \widetilde{\zeta}^{-\rho}>t\right) \sim \Gamma_\rho
		\E_x\left( \widetilde{Z}_t^{-\rho}((0,\infty))\right),
	    \]
	    where $\Gamma_\rho = C_{sub}$ when $\rho\leq 0$ and $\Gamma_\rho= \rho^2 C_0(\rho)$ when $\rho>0$.
 \end{theorem}
 \begin{remark}\label{remark}
Combining Theorem \ref{thm1} with the the asymptotic behavior of $\mathbf{P}_x^{-\rho}(\tau_0>t)$ (where, for any $y\in \R$, $\tau_y$ is the first hitting time of $y$), we see that,
when $\rho\leq 0$, $\P_x\left( \widetilde{\zeta}^{-\rho}>t\right)\sim \P_x\left(\zeta>t\right) \mathbf{P}_x^{-\rho}(\tau_0>t)$, i.e.,
the branching and the spatial motion are nearly independent. However, due to the appearance of $C_0(\rho)$, when $\rho>0$, the branching and the spatial motion are not nearly independent.
 \end{remark}

Our second main result is on the tail probability
$\mathbb{P}_x(\widetilde{M}^{-\rho}>y)$. In the case  when there is no killing,
the results \eqref{tail-prob-sub-M} and \eqref{Generalized-Tail-of-M} were proved under the assumption that the offspring distribution has finite third moment.
Our assumption \eqref{LLogL-moment-condition} on the offspring distribution is much weaker.

 \begin{theorem}\label{thm2}
	Assume that \eqref{LLogL-moment-condition} holds.  Then for any $\rho\in \R$,  there exists a constant $C_*(\rho)\in (0,\infty)$ such that for any $x>0$,
	\begin{align}
		\lim_{y\to\infty} e^{(\rho +\sqrt{2\alpha+\rho^2}) y} \P_x(\widetilde{M}^{-\rho}>y)= 2 C_*(\rho)e^{\rho x}\sinh(x\sqrt{2\alpha+\rho^2}).
	\end{align}
 \end{theorem}
 \begin{remark}
On $\{\widetilde{M}^{-\rho}> y\}$, there is at least one particle which achieves the level $y$ before hitting $0$.
The reason
for the appearance of  the $\sinh$ function
in the theorem above is that
this
function is related to  the Laplace transformation of $\tau_y$ on the event $\{\tau_y<\tau_0\}$ and this event gives the main contribution to the tail probability of $\{\widetilde{M}^{-\rho}> y\}$.
 \end{remark}

Our third main result is on the limit behavior of the maximal position at time $t$.

 \begin{theorem}\label{thm3}
Suppose that \eqref{LLogL-moment-condition} holds and $x>0$.
	\begin{enumerate}
		\item[(i)] For $\rho =0$ and $z\ge 0$,
		\begin{align}
			\lim_{t\to\infty} \sqrt{t}e^{\alpha t}\P_x\left(\widetilde{M}_t^{-\rho}> \sqrt{t}z\right)=\sqrt{\frac{2}{\pi}}C_{sub} xe^{-z^2/2},
		\end{align}
		or equivalently, as $t\to\infty$,
		\begin{align}
			\P_x\left(\widetilde{M}_t^{-\rho}> \sqrt{t}z\right)\sim C_{sub}
			\E_x\left(\widetilde{Z}_t^{-\rho}	\left( (\sqrt{t}z,\infty) \right)\right).
		\end{align}
		\item[(ii)] For $\rho<0$ and $z\in \mathbb{R}$,
		\begin{align}
			\lim_{t\to\infty} e^{\alpha t}\P_x\left(\widetilde{M}_t^{-\rho}+\rho t> \sqrt{t}z\right)=\frac{C_{sub}(1-e^{2\rho x})}{\sqrt{2\pi}}\int_{z}^{\infty}e^{-\frac{y^2}{2}} \mathrm{d} y,
		\end{align}
		or equivalently, as $t\to\infty$,
		\begin{align}
			\P_x\left(\widetilde{M}_t^{-\rho}+\rho t> \sqrt{t}z\right)\sim C_{sub}
            \E_x\left(	\widetilde{Z}_t^{-\rho}	\left( (\sqrt{t}z-\rho t,\infty)\right)\right).
		\end{align}
		\item[(iii)] For $\rho>0$	and $z\geq 0$,
		\begin{align}
			\lim_{t\to\infty}t^{\frac{3}{2}} e^{(\alpha+\frac{\rho^2}{2})t}\P_x\left(\widetilde{M}_t^{-\rho}> z\right)=\sqrt{\frac{2}{\pi}}C_z(\rho)xe^{\rho x},
		\end{align}
	where
	$C_z(\rho):=\lim_{N\to\infty}	 e^{(\alpha+\frac{\rho^2}{2})N}	\int_{0}^{\infty}ye^{-\rho y} \P_y\left(\widetilde{M}_N^{-\rho}> z\right) \mathrm{d}y\in (0,\infty)$
	is a function of $z$
	independent of $x$.	Or equivalently, as $t\to\infty$,
	%\begin{align}
	\begin{align}\label{e:thrm1.3iiib}
		\P_x\left(\widetilde{M}_t^{-\rho}> z\right)\sim
		\frac{\rho^2 C_z(\rho)e^{\rho z}}{\rho z+1}
		 \E_x\left(	\widetilde{Z}_t^{-\rho}	\left( (z,\infty)\right)\right).
	\end{align}
	\end{enumerate}
 \end{theorem}

Combining Theorems \ref{thm1} and  \ref{thm3}, we get the following Yaglom-type theorem:
 \begin{corollary}\label{cor1}
	Suppose that \eqref{LLogL-moment-condition} holds and $x>0$.
	\begin{enumerate}
	\item[(i)]
	If $\rho =0$, then
 	\begin{align}
 		\P_x\left(\frac{\widetilde{M}_t^{-\rho}}{\sqrt{t}}\in \cdot \big| \widetilde{\zeta}^{-\rho}>t \right) \quad \stackrel{\mathrm{d}}{\Longrightarrow} \quad \P(R\in \cdot),
 	\end{align}
 	where $(R, \P)$ is a Rayleigh distribution with density $ze^{-z^2/2}1_{\{z>0\}}$.

     \item[(ii)]
     If $\rho<0$, then
     \begin{align}
     	\P_x\left(\frac{\widetilde{M}_t^{-\rho}+\rho t}{\sqrt{t}}\in \cdot \big| \widetilde{\zeta}^{-\rho}>t \right) \quad \stackrel{\mathrm{d}}{\Longrightarrow} \quad \mathbf{P}_0(B_1\in \cdot),
     \end{align}
     where $(B_1, \mathbf{P}_0)$
       is a standard normal distribution.

     \item[(iii)]
     If $\rho>0$, then there exists a random variable $(X,\P)$  whose law is independent of $x$ such that
       \begin{align}
     	\P_x\left(\widetilde{M}_t^{-\rho}\in \cdot \big| \widetilde{\zeta}^{-\rho}>t \right) \quad \stackrel{\mathrm{d}}{\Longrightarrow} \quad\P (X\in \cdot).
     \end{align}
     \end{enumerate}
 \end{corollary}
\begin{remark}
	Compared with \cite[Theorem 3]{Lalley15} in the case of
	critical branching random walks, for $\rho \leq 0$, the weak limit of $\widetilde{M}_t^{-\rho}$ conditioned on survival up to time $t$ is simpler. The limit in \cite[Theorem 3]{Lalley15} is related to the maximum of a measure-valued process (see \cite[Corollary 4]{Lalley15}).
\end{remark}

\begin{remark}
	It is natural to study similar problems for subcritical
	branching killed L\'{e}vy processes.
	However, in the general case, even when the spatial motion is spectrally negative, some of the main ingredients, such as Lemma \ref{lemma_tau>t_with_drift}, are much more difficult.
	So, to avoid technical details, we	concentrate on the case of subcritical branching killed Brownian motion with drift.
\end{remark}

{\bf Organization of the paper:}
The rest of the paper is organized as follows. In Section  \ref{subsection_BM}, we first give some results on Brownian motion and the 3-dimensional Bessel process
that will be used in the proofs of our main results. Then we recall
some connections between the one-sided F-KPP equation and
our model in Section \ref{subsection_oneside_FKPP}.
The proofs of Theorems \ref{thm1} and \ref{thm3} are given in Section \ref{proof of thm1}
and the proof of Theorem \ref{thm2} is given in Section \ref{proof of thm2}.

\section{Preliminaries}\label{Preliminary}

 \subsection{Some useful properties of Brownian motion}\label{subsection_BM}

Let $(B_t,\mathbf{P}_x)$ be a standard Brownian motion
starting from $x$. For any $\rho\in \R$, it is known that
$\{e^{-\rho (B_t-x)-\frac{\rho^2}{2}t},t\ge0\}$ is a positive $\mathbf{P}_x$-martingale with mean $1$.
Define $\mathcal{F}_t:=\sigma(B_s:s\le t)$ and
\begin{align}\label{change_mea_BMwith_drift}
	\frac{\mathrm{d}\mathbf{P}_x^{-\rho} }{\mathrm{d}\mathbf{P}_x}\Big|_{
	\mathcal{F}_t}	:=e^{-\rho (B_t-x)-\frac{\rho^2}{2}t}.
\end{align}
Then under $\mathbf{P}_x^{-\rho}$, $\{B_t,t\ge 0\}$ is a Brownian motion with drift $-\rho$ starting from $x$.  For any $z\in \R$, define $\tau_z:=\inf\{t> 0:B_t=z\}$. Note that for any $x>0$, under $\mathbf{P}_x$, $\frac{B_t}{x}1_{\{\tau_0>t\}}$ is a 
positive martingale of mean $1$.  Define
\begin{align}\label{change_mea}
	\frac{\mathrm{d}\mathbf{P}_x^B }{\mathrm{d}\mathbf{P}_x}\Big|_{
	\mathcal{F}_t}	:=\frac{B_t}{x}1_{\{\tau_0>t\}}=\frac{B_t}{x}1_{\{\min_{s\le t}B_s>0\}}.
\end{align}
It is well-known that $(B_t, \mathbf{P}_x^B)$ is a $3$-dimensional Bessel process with transition probability density $p_t^B(x,y)$ given by
\begin{align}
	p_t^B(x,y):=
	\frac{y}{x}\frac{1}{\sqrt{2\pi t}}e^{-\frac{(y-x)^2}{2t}} \left(1-e^{-\frac{-2xy}{t}}\right)1_{\{y>0\}}.
\end{align}

The following result gives the asymptotic behavior of
$\mathbf{P}_x^{-\rho}(\tau_0>t, B_t>z(t,\rho))$  as $t\to\infty$ where $z(t,\rho)= \sqrt{t}z-\rho t$ for $\rho\leq 0$ and $z(t,\rho)=z$ for $\rho>0$.
For the case $\rho<0$, see \cite[page 30]{Borodin02} and for the case $\rho>0$,
see \cite[(7) and Lemma 3.1]{Louidor20}. The case for $\rho=0$ is easy to deal with using \eqref{change_mea}, so we omit the proof.

\begin{lemma}\label{lemma_tau>t_with_drift}
	Let $x>0$.
	\begin{enumerate}
		\item[(i)]	If $\rho=0$, then for any $z\geq 0$, we have
		\begin{align}
			\lim_{t\to \infty} \sqrt{t}\mathbf{P}_x(\tau_0>t, B_t>\sqrt{t}z)=\sqrt{\frac{2}{\pi}} x e^{-\frac{z^2}{2}}.
		\end{align}
		\item[(ii)]	 If $\rho <0$, then
		\begin{align}
			\lim_{t\to\infty}\mathbf{P}_x^{-\rho}(\tau_0>t)=1-e^{2\rho x}.
		\end{align}
		Also, for any $z\in \R$,
		\begin{align}
				\lim_{t\to\infty}\mathbf{P}_x^{-\rho}\left(\tau_0>t,B_t+\rho t>\sqrt{t}z\right)
			=\frac{(1-e^{2\rho x})}{\sqrt{2\pi}}\int_{z}^{\infty}e^{-\frac{y^2}{2}} \mathrm{d} y.
		\end{align}
		\item[(iii)]	If $\rho >0$, then for any $z\geq 0$,
		\begin{align}\label{asymptotic_tua>t_rho>0}
			\lim_{t\to\infty} t^{\frac{3}{2}}e^{\frac{\rho^2}{2}t}\mathbf{P}_x^{-\rho}(\tau_0>t, B_t>z)
			=\sqrt{\frac{2}{\pi}} xe^{\rho x}\int_{z}^{\infty}ye^{-\rho y} \mathrm{d}y.
		\end{align}
		Consequently, for any $A\subset (0,\infty)$ with $|\partial A|=0$,
		\[
			\lim_{t\to\infty}\mathbf{P}_x^{-\rho} (B_t\in A |\tau_0>t)= \rho^2 \int_A ye^{-\rho y} \mathrm{d}y.
		\]
	\end{enumerate}
\end{lemma}

In the following result, we give the asymptotic behaviors of
$\E_x\left( \widetilde{Z}_t^{-\rho}((0,\infty))\right)$ and $\E_x\left(\widetilde{Z}_t^{-\rho}\left((z(t,\rho),\infty)\right) \right)$ as $t\to\infty$.
\begin{lemma}\label{lemma-expectation-survival-number}
	Let $x>0$.
	\begin{enumerate}
		\item[(i)] If $\rho=0$,	then for any $z\geq 0$,
		\begin{align}
			\lim_{t\to\infty}\sqrt{t}e^{\alpha t}
						\E_x\left( \widetilde{Z}_t^{-\rho} ((\sqrt{t}z,\infty)) \right)
			=\sqrt{\frac{2}{\pi}}x e^{-\frac{z^2}{2}}.
		\end{align}
		\item[(ii)] If $\rho<0$, we have
		\begin{align}
			\lim_{t\to\infty}e^{\alpha t}\E_x
						\left( \widetilde{Z}_t^{-\rho}((0,\infty)) \right)=1-e^{2\rho x},
		\end{align}
		and for any $z\in \R$,
		\begin{align}
			\lim_{t\to\infty}e^{\alpha t}\E_x
					\left(\widetilde{Z}_t^{-\rho}((\sqrt{t}z-\rho t,\infty))\right)
			=\frac{1-e^{2\rho x}}{\sqrt{2\pi}}\int_{z}^{\infty}e^{-\frac{y^2}{2}}\mathrm{d}y.
		\end{align}
		\item[(iii)] If $\rho>0$, then for any $z\geq 0$,
		\begin{align}
			\lim_{t\to\infty}t^{3/2}e^{(\alpha+\frac{\rho^2}{2})t}\E_x
						\left(\widetilde{Z}_t^{-\rho}((z,\infty))\right)
			=\sqrt{\frac{2}{\pi}}x e^{\rho x} \int_{z}^{\infty}ye^{-\rho y} \mathrm{d}y
			= \frac{1}{\rho^2}\sqrt{\frac{2}{\pi}} xe^{\rho (x-z)} (\rho z+1).
		\end{align}
	\end{enumerate}	
\end{lemma}
\begin{proof}
		For any bounded measurable function $F$,
		by the many-to-one lemma (see Hardy and Harris \cite[Theorem 2.8]{Hardy06}),
	we have
	\begin{align}\label{Many-to-one}
		\E_x\Big(\sum_{u\in N_t^{-\rho}} F(X_u(s), 0\leq s\leq t)\Big) = e^{-\alpha t}\mathbf{E}_x^{-\rho}\left( F(B_s, 0\leq s\leq t)\right),
	\end{align}
	which implies that
	\begin{align}
			 \E_x\left(\widetilde{Z}_t^{-\rho}((0,\infty))\right)
		=e^{-\alpha t}\mathbf{P}_x^{-\rho} (\tau_0>t)
	\end{align}
	and
	\begin{align}
			\E_x\left(\widetilde{Z}_t^{-\rho}((z(t,\rho),\infty))\right)
		=e^{-\alpha t}\mathbf{P}_x^{-\rho} (B_t>z(t,\rho), \tau_0>t).
	\end{align}
	Combining this with Lemma \ref{lemma_tau>t_with_drift},	we arrive at the desired result.
\end{proof}

For $x,y>0$, define $v(x,y):=\P_x(\widetilde{M}^{-\rho}>y)$.
Lemma \ref{lemma_Bessel_driftBM} below
will play an important role
in the proof of Theorem \ref{thm2}. To prove this result, 
%we prove two other lemmas first.
we give two elementary lemmas first. The proofs of these two lemmas are routine and we give
the details for completeness.

\begin{lemma}\label{lemma_Bessel_BM}
	For any $a\ge 0$, $0<x\le y$ and nonnegative Borel function $h$, we have
	\begin{align}
 \mathbf{E}_x\left(1_{\{\tau_y<\tau_0\}}e^{-a\tau_y-\int_{0}^{\tau_y}h(B_s) \mathrm{d}s}\right)=\frac{x}{y}\mathbf{E}_x^B\left(e^{-a\tau_y-\int_{0}^{\tau_y} h(B_s) \mathrm{d}s}\right).
	\end{align}
\end{lemma}
\begin{proof}
	Note that $\mathbf{P}_x^B(\tau_y=\infty)=0$ for any $x\le y$. Since $\mathcal{F}_{\tau_y\land t}\subset\mathcal{F}_{t}$, it follows from 	\eqref{change_mea}	 that
	\begin{align}
		&\mathbf{E}_x^B\left(e^{-a\tau_y-\int_{0}^{\tau_y} h(B_s)\mathrm{d}s} \right) =
		\lim_{t\to\infty}\mathbf{E}_x^B\left(
		e^{-a\tau_y-\int_{0}^{\tau_y} h(B_s)
		\mathrm{d}s}1_{\{\tau_y<t\}}\right)\\
		&=\lim_{t\to\infty}\mathbf{E}_x\left(\frac{B_t}{x}1_{\{\tau_0>t\}}	
		e^{-a\tau_y-\int_{0}^{\tau_y} h(B_s)	
		\mathrm{d}s}1_{\{\tau_y<t\}}\right)\\
		&=\lim_{t\to\infty}\mathbf{E}_x\left(
		e^{-a\tau_y-\int_{0}^{\tau_y} h(B_s)
		\mathrm{d}s}1_{\{\tau_y<t\}}\mathbf{E}_x\left(\frac{B_t}{x}1_{\{\tau_0>t\}}|\mathcal{F}_{\tau_y\land t}\right)\right).
	\end{align}
	Since $(\frac{B_t}{x}1_{\{\tau_0>t\}})_{t\ge 0}$ is a $\mathbf{P}_x$-martingale with respect to $(\mathcal{F}_t)_{t\ge 0}$, by the optional stopping theorem, we have
	\begin{align}
		\mathbf{E}_x	\left(\frac{B_t}{x}1_{\{\tau_0>t\}}|\mathcal{F}_{\tau_y\land t}\right)
		=\frac{B_{\tau_y\land t}}{x}1_{\{\tau_0>{\tau_y\land t}\}}.
	\end{align}
	It follows from the dominated convergence theorem that
	\begin{align}
		&\mathbf{E}_x^B\left(e^{-a\tau_y-\int_{0}^{\tau_y} h(B_s)		\mathrm{d}s}\right)
		=\frac{y}{x}\lim_{t\to\infty}\mathbf{E}_x\left( 1_{\{\tau_y<t,\tau_0>\tau_y\}}e^{-a\tau_y-\int_{0}^{\tau_y} h(B_s)
		\mathrm{d}s}\right)\\
		&=\frac{y}{x}\mathbf{E}_x\left(1_{\{\tau_0>\tau_y\}}e^{-a\tau_y-\int_{0}^{\tau_y} h(B_s)\mathrm{d}s}\right).
	\end{align}
	This completes the proof.
\end{proof}

\begin{lemma}\label{lemma_BM_driftBM}
	 For any $a\ge 0$, 	$0<x\le y$ and non-negative Borel function $h$, we have
	\begin{align}
		 \mathbf{E}_x^{-\rho}\left(1_{\{\tau_y<\tau_0\}}e^{-a\tau_y-\int_{0}^{\tau_y} h(B_s)\mathrm{d}s}\right)
		= e^{\rho (x-y)}	\mathbf{E}_x\left(1_{\{\tau_y<\tau_0\}}e^{-(a+\frac{\rho^2}{2}) \tau_y-\int_{0}^{\tau_y} h(B_s)\mathrm{d}s}\right).
	\end{align}
\end{lemma}
\begin{proof}
We deal with the case $a>0$ first. For $a>0$,  since $e^{-a\tau_y} 1_{\{\tau_y=\infty\}}=0$,	it follows from \eqref{change_mea_BMwith_drift} that
	\begin{align}\label{change_BM_driftBM}
		& \mathbf{E}_x^{-\rho}\left(e^{-a\tau_y-\int_{0}^{\tau_y} h(B_s)	\mathrm{d}s}\right) \\ &	=
		\lim_{t\to\infty}\mathbf{E}_x^{-\rho}\left(
		e^{-a\tau_y-\int_{0}^{\tau_y} h(B_s)
		\mathrm{d}s}1_{\{\tau_y<t\}}\right)\nonumber\\
		&=\lim_{t\to\infty}\mathbf{E}_x\left(	e^{-\rho(B_t-x)-\frac{\rho^2}{2}t}
		e^{-a\tau_y-\int_{0}^{\tau_y} h(B_s)
		\mathrm{d}s}1_{\{\tau_y<t\}}\right)\nonumber\\
		&=\lim_{t\to\infty}\mathbf{E}_x\left( e^{-a\tau_y-\int_{0}^{\tau_y} h(B_s)\mathrm{d}s}1_{\{\tau_y<t\}}\mathbf{E}_x\left(
		e^{-\rho(B_t-x)-\frac{\rho^2}{2}t}	|{\mathcal{F}_{\tau_y\land t}}\right)\right).\nonumber
	\end{align}
	Recall that  $(e^{-\rho(B_t-x)-\frac{\rho^2}{2}t})_{t\ge0}$	 is a $\mathbf{P}_x$-martingale with respect to $(\mathcal{F}_t)_{t\ge 0}$, 	 so by the optional stopping theorem, on $\{\tau_y <t\}$, we have
	\begin{align}
		\mathbf{E}_x\left( e^{-\rho(B_t-x)-\frac{\rho^2}{2}t}|{\mathcal{F}_{\tau_y\land t}}\right)		=e^{-\rho(B_{\tau_y\land t}-x)-\frac{\rho^2}{2}(\tau_y\land t)}	= e^{-\rho(y-x)-\frac{\rho^2}{2}\tau_y}.
	\end{align}
	Combining this with  \eqref{change_BM_driftBM}	and using the fact that $\mathbf{P}_x\left(\tau_y<\infty\right)=1$, we get
	\begin{align}\label{expression_BM_driftBM}
		    \mathbf{E}_x^{-\rho}\left( e^{-a \tau_y-\int_{0}^{\tau_y} h(B_s)\mathrm{d}s}	\right)		=e^{\rho(x-y)}
		\mathbf{E}_x\left( e^{-(a+\frac{\rho^2}{2})\tau_y-\int_{0}^{\tau_y} h(B_s)\mathrm{d}s}\right).
	\end{align}
	Similarly, for $a>0$, we have
	\begin{align}\label{change_BM_driftBM1}
		&\mathbf{E}_x^{-\rho}\left( 1_{\{\tau_y\ge \tau_0\}}e^{-a\tau_y-\int_{0}^{\tau_y} h(B_s)\mathrm{d}s}\right)	 \\
		&=\lim_{t\to\infty}\mathbf{E}_x^{-\rho}\left(
		1_{\{\tau_y\ge \tau_0\}}	e^{-a\tau_y-\int_{0}^{\tau_y} h(B_s)
		\mathrm{d}s}1_{\{\tau_y<t\}}\right)\nonumber\\
		&=\lim_{t\to\infty}\mathbf{E}_x\left( 	e^{-\rho (B_t-x)-\frac{\rho^2}{2}t}
		1_{\{\tau_y\ge \tau_0\}}
		e^{-a\tau_y-\int_{0}^{\tau_y}
		 h(B_s)\mathrm{d}s}1_{\{\tau_y<t\}}\right)\nonumber\\
		&=\lim_{t\to\infty}\mathbf{E}_x\left(
		1_{\{\tau_y\ge \tau_0\}}
		e^{-a\tau_y-\int_{0}^{\tau_y}	 h(B_s)
		\mathrm{d}s}1_{\{\tau_y<t\}}
		e^{-\rho(B_{\tau_y\land t}-x)-\frac{\rho^2}{2}(\tau_y\land t)}
		\right)\nonumber\\	&=	e^{\rho (x-y)}	\mathbf{E}_x\left( 1_{\{\tau_y \ge \tau_0\}}e^{-(a+\frac{\rho^2}{2})\tau_y-\int_{0}^{\tau_y} h(B_s)\mathrm{d}s}		\right),\nonumber
	\end{align}
	where in the last inequality we used $\mathbf{P}_x\left(\tau_y<\infty \right)=1$.
	Combining \eqref{expression_BM_driftBM} and \eqref{change_BM_driftBM1},  we arrive at the desired result for $a>0$.
		
	For the case $a=0$, by the dominated convergence theorem, we have
	\begin{align}
		&	\mathbf{E}_x^{-\rho}\left( 1_{\{\tau_y<\tau_0\}}e^{-\int_{0}^{\tau_y}
		h(B_s)\mathrm{d}s}\right)
		= \lim_{\theta \to0+} \mathbf{E}_x^{-\rho}\left( 1_{\{\tau_y<\tau_0\}}e^{ -\theta \tau_y -\int_{0}^{\tau_y}
		h(B_s)\mathrm{d}s}\right) \nonumber\\
		& =	\lim_{\theta\to0+}	e^{\rho(x-y)}	\mathbf{E}_x\left( 1_{\{\tau_y<\tau_0\}}e^{-(\theta +\frac{\rho^2}{2})\tau_y-\int_{0}^{\tau_y}
		h(B_s)\mathrm{d}s}\right)\nonumber\\
		& = e^{\rho(x-y)}	\mathbf{E}_x\left( 1_{\{\tau_y<\tau_0\}}e^{-\frac{\rho^2}{2}\tau_y-\int_{0}^{\tau_y}
		h(B_s)\mathrm{d}s}\right).
	\end{align}
	This completes the proof.
\end{proof}

Combining Lemmas \ref{lemma_Bessel_BM} and \ref{lemma_BM_driftBM}, we immediately get the following result.

\begin{lemma}\label{lemma_Bessel_driftBM}
	For any $a\ge 0$, $0<x\le y$ and nonnegative Borel
	function $h$, we have
	\begin{align}
		\mathbf{E}_x^{-\rho}\left( 1_{\{\tau_y<\tau_0\}}e^{-a\tau_y-\int_{0}^{\tau_y} h(B_s)\mathrm{d}s} \right)
		=\frac{x}{y}	e^{\rho(x-y)}	\mathbf{E}_x^B\left( e^{-(a+\frac{\rho^2}{2})\tau_y-\int_{0}^{\tau_y} h(B_s)\mathrm{d}s}\right).
	\end{align}
\end{lemma}

The following result can be found on \cite[page 469]{Borodin02}.
\begin{lemma}\label{lemma_property_Bessel3}
	For any $a>0$ and $0<x\le y$, it holds that
	\begin{align}
		\mathbf{E}_x^B	\left( e^{-a\tau_y}\right)=\frac{y\sinh(x\sqrt{2a})}{x\sinh(y\sqrt{2a})}.
	\end{align}
\end{lemma}
Combining Lemmas \ref{lemma_Bessel_driftBM} and  \ref{lemma_property_Bessel3},
we see that for any $\rho>0$ and $x>0$,
\begin{align}\label{Scale-function}
     \lim_{y\to\infty}	\mathbf{P}_x^\rho\left(\tau_y<\tau_0\right) = \lim_{y\to\infty} e^{\rho(y-x)}  \frac{\sinh(x\rho )}{\sinh(y\rho)}=  1-e^{-2\rho x}.
\end{align}

\subsection{One side F-KPP equation}\label{subsection_oneside_FKPP}

According to Chauvin and Rouault \cite{Chauvin88}, a branching killed Brownian motion with drift is closely related to the following PDE
\begin{align}\label{one-side-FKPP}
	 \frac{\partial w}{\partial t}=\frac{1}{2}\frac{\partial^2w}{\partial x^2}-\rho\frac{\partial w}{\partial x}+	\beta \left(\sum_{k=0}^{\infty}p_k w^k-w\right)
\end{align}
on $[0,\infty)\times [0,\infty)$.
Let $\widetilde{N}_t^{-\rho}$ be the set of particles alive at time $t$ of the branching killed Brownian motion.
Then for any $[0,1]$-valued function $h$ on
$[0, \infty)$ with $h(0)=1$,
$w(x,t)=\mathbb{E}_x(\prod_{v\in \widetilde{N}_t^{-\rho}}h(X_v(t)))$ is a solution of \eqref{one-side-FKPP} with initial  condition $w(x,0)=h(x)$.
Define $w(x,t):=\mathbb{P}_x(\widetilde{\zeta}^{-\rho}\le t)$ and let $s\in [0,t]$. By the Markov property, we have
\begin{align}
	w(x,t)=\mathbb{E}_x\left(\prod_{v\in \widetilde{N}_s^{-\rho}}w(X_u(s),t-s)\right).
\end{align}
Thus, $w(x,t)$ is a solution to \eqref{one-side-FKPP} with initial  condition
 $w(x,0)=1_{\{x\le 0\}}$ and boundary condition $w(0+,t)=1$. Now let
\begin{align}\label{Def-u}
	u(x,t):=\mathbb{P}_x(\widetilde{\zeta}^{-\rho}>t)=1-w(x,t).
\end{align}
Then  $u$ satisfies
\begin{equation}\label{PDE_u}
	\frac{\partial u}{\partial t}=\frac{1}{2}\frac{\partial^2u}{\partial x^2}-\rho\frac{\partial u}{\partial x}-\Phi(u)
 \quad \mbox{ on } (0,\infty)\times (0,\infty)
\end{equation}
 with initial condition $u(x,0)=1_{(0,\infty)}(x)$ and boundary condition $u(0+,t)=0$, where the function $\Phi$ is defined in \eqref{def_Phi}. Similarly, for any $z>0$, the function
\begin{equation}\label{def-Q}
    Q_z(x,t):=\mathbb{P}_x(\widetilde{M}_t^{-\rho}>z),\quad x,t>0
\end{equation}
 satisfies \eqref{PDE_u} with initial condition $Q_z(x,0)=1_{\{x>z\}}$ and boundary condition $Q_z(0+,t)=0$.

The next simple result will be used
in the proofs of our main results.

\begin{lemma}\label{upper-k}
 The function $\varphi(u)$ is increasing in
  $u\in [0,1]$.
 Moreover, under \eqref{LLogL-moment-condition}, for any $c>0$, it holds that
 \[
      \int_0^\infty \varphi\left(e^{-ct }\right)\mathrm{d}t<\infty.
 \]
\end{lemma}
\begin{proof}
	 By the definition of $\varphi$,
	\begin{align}
		&\beta ^{-1} \varphi(u)
		 = \frac{\sum_{k=0}^\infty p_k(1-u)^k -(1-u)}{u} -\left(1-\sum_{k=0}^\infty kp_k\right)\nonumber\\
		& =  \sum_{\ell=0}^\infty \left(\sum_{k=\ell+1}^\infty p_k\right)-\sum_{k=1}^\infty p_k \sum_{\ell=0}^{k-1} (1-u)^\ell  =  \sum_{\ell=0}^\infty \left(\sum_{k=\ell+1}^\infty p_k\right) \left( 1-(1-u)^\ell\right).
	\end{align}
	Therefore, $\varphi$ is increasing in $u$.
	Combining
	 the monotonicity of $\varphi$ and 	\eqref{e:rsnew}, we have
    \begin{align}
    \int_0^\infty \varphi(C_{sub}e^{-\alpha t})\mathrm{d}t \leq 	\int_0^\infty \varphi(g(t))\mathrm{d}t<\infty.
    \end{align}
  Setting $N:= - \frac{1}{\alpha}\log C_{sub} $, then for any $c>0$,
    \begin{align}
    	 & \int_0^\infty \varphi\left(e^{-ct }\right)\mathrm{d}t = \frac{\alpha}{c} \int_0^\infty \varphi\left(e^{-\alpha t }\right)\mathrm{d}t \leq \frac{\alpha}{c}\int_0^N \varphi(1)\mathrm{d}t+ \frac{\alpha}{c}\int_0^\infty \varphi(e^{-\alpha(t-N)})\mathrm{d}t \\
    	 &= \frac{\alpha}{c} N \varphi(1) + \frac{\alpha}{c}  \int_0^\infty \varphi(C_{sub}e^{-\alpha t})\mathrm{d}t<\infty.
    \end{align}
\end{proof}

\section{Proofs of Theorem \ref{thm1} and Theorem \ref{thm3}}\label{proof of thm1}

 In this section, we prove Theorem \ref{thm1} and Theorem \ref{thm3} by establishing some upper and lower bounds for the functions $u(t,x)$  and $Q_{z}(x,t)$ defined in \eqref{Def-u} and \eqref{def-Q} respectively.
It is easy to see that
\begin{align}\label{Identity}
Q_{0}(x,t)= \P_x\left(\widetilde{M}^{-\rho}_t> 0\right)= \P_x \left(\widetilde{\zeta}^{-\rho}>t\right) = u(x,t).
\end{align}

We first estimate $Q_{\sqrt{t}z-\rho t}(x,t)$ and $u(x,t)$  from below.
We treat the cases  $\rho=0$ and $\rho<0$ together since it turns out that branching and spatial motion are nearly independent in these two cases.

\begin{lemma}\label{lemma_upper_bound_u}
	Suppose that $x>0$ and $\rho\leq 0$.
	\begin{enumerate}
		\item[(i)] If $\rho=0$, then for any $z\geq 0$,
		\[
		\liminf_{t\to\infty}\sqrt{t}e^{\alpha t}Q_{\sqrt{t}z}(x,t) \ge  \sqrt{\frac{2}{\pi}} C_{sub} x e^{-\frac{z^2}{2}}.
		\]
		\item[(ii)] If $\rho <0$, then
		\[
			\liminf_{t\to\infty}e^{\alpha t}u(x,t)\ge  C_{sub} (1-e^{2\rho x}),
		\]
		and for any $z\in \R$,
		\[
			\liminf_{t\to\infty}e^{\alpha t}Q_{\sqrt{t}z-\rho t}(x,t)
			\geq \frac{C_{sub}(1-e^{2\rho x})}{\sqrt{2\pi}} \int_z^\infty e^{-\frac{y^2}{2}}\mathrm{d}y.
		\]
 	\end{enumerate}
\end{lemma}

\begin{proof}
	Recall that $Q_z(x,t)$ satisfies \eqref{PDE_u}
	with initial condition $Q_z(x,0)=1_{(z,\infty)}(x)$ for $x>0$ and boundary condition $Q_z(0+,t)=0$ for $t> 0$.
	Combining
	 the definition \eqref{def_Phi} of $\varphi$  and the  Feynman-Kac formula, we have
	\begin{align}\label{expression_u}
		Q_z(x,t)=e^{-\alpha t}\mathbf{E}_x^{-\rho} \left( 1_{\{\tau_0>t, B_t>z\}}e^{-\int_{0}^{t}\varphi(Q_z(B_s,t-s))\mathrm{d}s}
		\right).
	\end{align}
	Since $\widetilde{\zeta} \leq \zeta$, it holds that
	\begin{align}\label{upper_bound_u-2}
		Q_z(x,t)\le \mathbb{P}_x(\zeta >t)=g(t),\quad x, t>0, z\geq 0.
	\end{align}
Thus by Lemma \ref{upper-k},
	\begin{align}
		Q_z(x,t)	&\ge e^{-\alpha t}\mathbf{E}_x^{-\rho}\left( 1_{\{\tau_0>t, B_t>z\}}e^{-\int_{0}^{t}\varphi(g(t-s))\mathrm{d}s}\right)\\
		&=e^{-\int_{0}^{t}\varphi(g(s))\mathrm{d}s}e^{-\alpha t}\mathbf{P}_x^{-\rho}(\tau_0>t, B_t>z)\\
		&\ge C_{sub}e^{-\alpha t}	\mathbf{P}_x^{-\rho}(\tau_0>t, B_t>z),
	\end{align}
where in the last inequality we used \eqref{Constant-C-sub}. Recalling \eqref{Identity} and using Lemma \ref{lemma_tau>t_with_drift} with $z$ replaced by $0$ and $\sqrt{t}z -\rho t$, we get the desired result.
\end{proof}

In the following lemma, we
 give an upper bound
 of $\mathbb{P}_0(\max_{s\le t}M_s^{\rho}\ge x)$ which will be used to get the lower bound of $Q_z(x,t)$.

\begin{lemma}\label{lemma_lower_bound_u}
	Assume $\rho \le 0$.  For any $x,t>0$, it holds that
	\[
	 \mathbb{P}_0(\max_{s\le t}M_s^{\rho}\ge x)\leq e^{-\sqrt{2\alpha} x}.
	\]
\end{lemma}
\begin{proof}
By \eqref{Many-to-one}, it
is easily seen that $W_t^\rho :=\sum_{u\in N_t^{\rho}}e^{\sqrt{2\alpha}(X_u(t)-\rho t)}$ is a non-negative martingale. For any $\rho\leq 0$, by Doob's maximal inequality, we have
	\begin{align} \label{upper-bound-tail-of-maximal}
		\mathbb{P}_0(\max_{s\le t}M_s^{\rho}\ge x)		&\le	\mathbb{P}_0(\max_{s\le t}  e^{\sqrt{2\alpha}\rho s}W_s^\rho 	\ge e^{\sqrt{2\alpha} x})	\le \mathbb{P}_0(\max_{s\le t}
		W_s^\rho	\ge e^{\sqrt{2\alpha} x})	\\	&\le
		\frac{\mathbb{E}_0(W_t^\rho )}{e^{\sqrt{2\alpha} x}}	=e^{-\sqrt{2\alpha} x}.\nonumber
	\end{align}
   Therefore, we arrive at the desired result.
\end{proof}

\begin{lemma}\label{upper_bound_u}
	Assume that $\rho=0$ and $x>0$.  Then for any $z\geq 0$,  it holds that
		\begin{align}\label{limsup_u}
			\limsup_{t\to\infty}\sqrt{t}e^{\alpha t}Q_{\sqrt{t}z}(x,t) \le  \sqrt{\frac{2}{\pi}} C_{sub}x e^{-\frac{z^2}{2}}.
		\end{align}
\end{lemma}
\begin{proof}
	For any $y\geq x$,
	\begin{align}\label{Increasing-in-x-of-Q}
		&Q_z(y,t) = \P_y\left(\exists u\in N_t^{-\rho}\quad s.t.\ \min_{s\leq t}X_u(s)>0,\ X_u(t)>z\right)\\
		& \geq \P_y\left(\exists u\in N_t^{-\rho}\quad s.t.\ \min_{s\leq t}X_u(s)> y-x,\ X_u(t)>z+y-x\right)\nonumber\\
		& = \P_x\left(\exists u\in N_t^{-\rho}\quad s.t.\ \min_{s\leq t}X_u(s)>0,\ X_u(t)>z\right)= Q_z(x,t),
	\end{align}
	which implies that $Q_z(x,t)$ is increasing in $x$.	Fix an $N>0$. For $t\ge N$, by \eqref{Increasing-in-x-of-Q},
	\begin{align}\label{upper_bd_Q}
		Q_{\sqrt{t}z}(x,t)		&\le e^{-\alpha t} 		\mathbf{E}_x
		\left(  1_{\{\tau_0>t,B_t>\sqrt{t}z\}} e^{-\int_{t-N}^{t} \varphi(Q_{\sqrt{t}z}(B_s,t-s)) \mathrm{d}s}\right)
		\\
		&\le e^{-\alpha t} 	\mathbf{E}_x		\left( 1_{\{\tau_0>t,B_t>\sqrt{t}z\}} e^{-\int_{t-N}^{t} \varphi(Q_{\sqrt{t}z}(\inf_{r\in [t-N,t]}B_r,t-s)) \mathrm{d}s}\right)
		\nonumber\\
		&= e^{-\alpha t} 		\mathbf{E}_x		\left( 1_{\{\tau_0>t,B_t>\sqrt{t}z\}} e^{-\int_{0}^{N} \varphi(Q_{\sqrt{t}z}(\inf_{r\in [t-N,t]}B_r,s)) \mathrm{d}s}\right).
		 \nonumber
	\end{align}
	Take a $\gamma\in(0,\frac{1}{2})$ and define
	\begin{align}
		B_1(t):=\mathbf{E}_x\left( 1_{\{\tau_0>t,B_t>\sqrt{t}z, \inf_{r\in [t-N,t]}B_r\ge \sqrt{t}z+t^{\gamma} \}} e^{-\int_{0}^{N} \varphi(Q_{\sqrt{t}z}(\inf_{r\in [t-N,t]}B_r,s)) \mathrm{d}s}\right),
	\end{align}
	\begin{align}
		B_2(t):=\mathbf{E}_x\left( 1_{\{\tau_0>t,B_t>\sqrt{t}z, \inf_{r\in [t-N,t]}B_r< \sqrt{t}z+t^{\gamma} \}} e^{-\int_{0}^{N} \varphi(Q_{\sqrt{t}z}(\inf_{r\in [t-N,t]}B_r,s)) \mathrm{d}s}\right).
	\end{align}
	Then $Q_{\sqrt{t}z}(x,t)\le e^{-\alpha t}(B_1(t)+B_2(t))$.	Since $Q_z(x,t)$ is increasing in $x$, we have
	\begin{align}\label{B_1}
		B_1(t)
		\le &e^{-\int_{0}^{N} \varphi(Q_{\sqrt{t}z}( \sqrt{t}z+t^{\gamma},s)) \mathrm{d}s} \mathbf{P}_x(\tau_0>t,B_t>\sqrt{t}z).
	\end{align}
	Set $\widetilde{M}_s:= \widetilde{M}_s^0$ and $M_s:= M_s^0$ for simplicity. For any $s\leq N$, we have
	\begin{align}
		Q_{\sqrt{t}z}(\sqrt{t}z+t^{\gamma},s)
		&\ge \mathbb{P}_{\sqrt{t}z+t^{\gamma}}(\widetilde{M}_s>\sqrt{t}z,\inf_{r\le s}\inf_{u\in N_r}X_u(r)>0)\\
		&=\mathbb{P}_{\sqrt{t}z+t^{\gamma}}(M_s>\sqrt{t}z)
		-\mathbb{P}_{\sqrt{t}z+t^{\gamma}}(M_s>\sqrt{t}z,\inf_{r\le s}\inf_{u\in N_r}X_u(r)\le 0)\\
		&\ge \mathbb{P}_{0}(M_s>-t^{\gamma})	-\mathbb{P}_{0}(\inf_{r\le s}\inf_{u\in N_r}X_u(r)\le -(\sqrt{t}z+t^{\gamma}))\\	&=\mathbb{P}_{0}(M_s>-t^{\gamma})
		-\mathbb{P}_{0}(\max_{r\le s}M_r\ge \sqrt{t}z+t^{\gamma}).
	\end{align}
	According to \eqref{Many-to-one},
	\begin{align}\label{lower_bd_M>t^{beta}}
		&\mathbb{P}_{0}(M_s>-t^{\gamma})
		\ge \mathbb{P}_{0}(\zeta>s, M_s>-t^{\gamma}) =\mathbb{P}_{0}(\zeta>s)
		-\mathbb{P}_{0}(\zeta>s, M_s\le -t^{\gamma})
		\\
		&\geq \mathbb{P}_{0}(\zeta>s) - \P_0\big(\sum_{u\in N_s} 1_{\{X_u(s) \leq t^{-\gamma}\}}\geq 1\big)
		\nonumber\\
		&\ge \mathbb{P}_{0}(\zeta>s)
		-e^{-\alpha s}\mathbf{P}_0(B_s\le -t^{\gamma}).\nonumber
	\end{align}
	Combining this with Lemma \ref{lemma_lower_bound_u}, we get
	\begin{align}
		Q_{\sqrt{t}z}(\sqrt{t}z+t^{\gamma},s)	\ge g(s)	-e^{-\alpha s}\mathbf{P}_0(B_s\le -t^{\gamma})
		-e^{-\sqrt{2\alpha}(\sqrt{t}z+t^{\gamma})}.
	\end{align}
	Plugging this into  \eqref{B_1} and applying the dominated convergence theorem, we get
	\begin{align}
		&\limsup_{t\to\infty}\frac{B_1(t)}{\mathbf{P}_x(\tau_0>t,B_t>\sqrt{t}z)} \nonumber\\
		&	\le \limsup_{t\to\infty}\exp\left\{-\int_{0}^{N} \varphi\left(\left(g(s)-e^{-\alpha s}\mathbf{P}_0(B_s\le -t^{\gamma})-e^{-\sqrt{2\alpha}(\sqrt{t}z+t^{\gamma})}\right)_+ \right) \mathrm{d}s\right\}\\
		&=e^{-\int_{0}^{N} \varphi(g(s)) \mathrm{d}s}.
	\end{align}
	Letting $N\to \infty$, we get
	\begin{align}\label{asym_prob}
		\limsup_{N\to\infty}  \limsup_{t\to\infty}\frac{B_1(t)}{\mathbf{P}_x(\tau_0>t,B_t>\sqrt{t}z)}
		\le e^{-\int_{0}^{\infty} \varphi(g(s)) \mathrm{d}s}	=C_{sub}<\infty.
	\end{align}
	Therefore,	applying Lemma \ref{lemma_tau>t_with_drift} (i), we get
	\begin{align}\label{asym_B1}
		\limsup_{N\to\infty}	\limsup_{t\to\infty} \sqrt{t}B_1(t) \le  \sqrt{\frac{2}{\pi}}C_{sub} xe^{-z^2/2}.
	\end{align}
	Next, we show that
	$\lim_{t\to\infty}\sqrt{t}B_2(t)=0$.
	For	$\delta >0$, it holds that
	\begin{align}\label{eq:upper-B}
		B_2(t)	&\le	\mathbf{P}_x	(\tau_0>t,B_t>\sqrt{t}z, \inf_{r\in [t-N,t]}B_r< \sqrt{t}z+t^{\gamma} )
		\\ &\le \mathbf{P}_x	\left(\tau_0>t,\sqrt{t}z<B_t<\sqrt{t}	(z+\delta) \right)\nonumber\\
		&\quad+	\mathbf{P}_x\left(B_t\ge \sqrt{t}	(z+\delta), 	\inf_{r\in [t-N,t]}B_r< \sqrt{t}z+t^{\gamma}\right).\nonumber
	\end{align}
Note that $e^{-u}(1-e^{-x})\leq x$ for all  $u, x>0$. Thus  by \eqref{change_mea}, we get
	\begin{align}\label{eq:upper-B-1}
		& \mathbf{P}_x	\left(\tau_0>t,\sqrt{t}z<B_t<\sqrt{t}(z+\delta)\right)
		=\mathbf{E}_x^B\left( \frac{x}{B_t}1_{\{ \sqrt{t}z<B_t<\sqrt{t}(z+\delta)\}}\right)\\	&=\int_{\sqrt{t}z}^{\sqrt{t}(z+\delta)}\frac{1}{\sqrt{2\pi t}}e^{-\frac{(y-x)^2}{2t}}(1-e^{-\frac{2xy}{t}}) \mathrm{d}y	\le\frac{\delta}{\sqrt{2\pi }}\frac{2x(z+\delta)}{\sqrt{t}}.\nonumber
	\end{align}
   For any $t\ge N$, by the reflection principle,  we have
	\begin{align}\label{B2_prob_upper_bd}
		&\mathbf{P}_x\left(B_t\ge \sqrt{t}(z+\delta),	\inf_{r\in [t-N,t]}B_r< \sqrt{t}z+t^{\gamma}\right)
		\le \mathbf{P}_0\left(\inf_{r\in[0,N]}B_r	<-\delta	\sqrt{t}+t^{\gamma}\right)	\\
		& = \mathbf{P}_0\left( |B_N| >	\delta	\sqrt{t}-t^\gamma \right).\nonumber
	\end{align}
	Combining \eqref{eq:upper-B}, \eqref{eq:upper-B-1} and \eqref{B2_prob_upper_bd}, letting $t\to\infty$ first and then $\delta \to0$,  we get
	\begin{align}
		\lim_{t\to\infty}\sqrt{t}B_2(t)=
		0.
	\end{align}
	Combining this with \eqref{upper_bd_Q} and \eqref{asym_B1}, we get the desired assertion.
\end{proof}

\begin{lemma}\label{lemma3.4}
	Assume that $x>0$ and $\rho<0$.
	\begin{enumerate}
		\item[(i)] It holds that
		\begin{align}
			\limsup_{t\to\infty} e^{\alpha t} u(x,t)\le C_{sub}(1-e^{2\rho x}).
		\end{align}
		\item[(ii)] For any $z\in \R$, we have
		\begin{align}
		\limsup_{t\to\infty} e^{\alpha t} Q_{\sqrt{t}z-\rho t}(x,t)	\le  \frac{C_{sub}(1-e^{2\rho x})}{\sqrt{2\pi}}\int_{z}^{\infty}e^{-\frac{y^2}{2}} \mathrm{d} y.
		\end{align}
	\end{enumerate}
\end{lemma}
\begin{proof}
We will prove (i) and (ii) in one stroke. For (i) we put $z_t=0$ and for (ii) we put $z_t=\sqrt{t}z-\rho t$.
Then taking $z=z_t$ in  \eqref{expression_u}, we get
	\begin{align}\label{upper_bd_Q_rho<0}
		Q_{z_t}(x,t)	&\le e^{-\alpha t} \mathbf{E}_x^{-\rho}\left( 1_{\{\tau_0>t,B_t>z_t\}} e^{-\int_{t-N}^{t} \varphi(Q_{z_t}(B_s,t-s)) \mathrm{d}s}\right)	\\
		&\le e^{-\alpha t} \mathbf{E}_x^{-\rho}\left( 1_{\{\tau_0>t,B_t>z_t\}} e^{-\int_{0}^{N} \varphi(Q_{z_t}(\inf_{r\in[t-N,t]}B_r,s)) \mathrm{d}s}\right).\nonumber
	\end{align}
	Take a $\gamma\in(0,\frac{1}{2})$ and define
	\begin{align}
		C_1(t):=\mathbf{E}_x^{-\rho}\left( 1_{\{\tau_0>t,B_t>z_t, \inf_{r\in [t-N,t]}B_r\ge z_t+t^{\gamma} \}} e^{-\int_{0}^{N} \varphi(Q_{z_t}(\inf_{r\in [t-N,t]}B_r,s)) \mathrm{d}s}\right),
	\end{align}
	\begin{align}
		C_2(t):=\mathbf{E}_x^{-\rho}\left( 1_{\{\tau_0>t,B_t>z_t, \inf_{r\in [t-N,t]}B_r< z_t+t^{\gamma} \}} e^{-\int_{0}^{N} \varphi(Q_{z_t}(\inf_{r\in [t-N,t]}B_r,s)) \mathrm{d}s}\right).
	\end{align}
	Then $Q_{z_t}(x,t)\le e^{-\alpha t}(C_1(t)+C_2(t))$.	Using \eqref{Increasing-in-x-of-Q}, we have
	\begin{align}\label{C_1}
		C_1(t)\le &e^{-\int_{0}^{N} \varphi(Q_{z_t}( z_t+t^{\gamma},s)) \mathrm{d}s} \mathbf{P}_x^{-\rho}(\tau_0>t,B_t>z_t).
	\end{align}
	For any $s\leq N$, similarly to \eqref{lower_bd_M>t^{beta}}, for $t$  large enough such that $z_t\geq 0$, we have
	\begin{align}
		Q_{z_t}(z_t+t^{\gamma},s) &\ge \mathbb{P}_{z_t+t^{\gamma}}(\widetilde{M}_s^{-\rho}>z_t,\inf_{r\le s}\inf_{u\in N_r^{-\rho}}X_u(r)>0)\\ &=\mathbb{P}_{z_t+t^{\gamma}}(M_s^{-\rho}> z_t)-\mathbb{P}_{z_t+t^{\gamma}}(M_s^{-\rho}>z_t,\inf_{r\le s}\inf_{u\in N_r^{-\rho}}X_u(r)\le 0)\\
		&\ge \mathbb{P}_{0}(M_s^{-\rho}>-t^{\gamma})
		-\mathbb{P}_{0}(\inf_{r\le s}\inf_{u\in N_r^{-\rho}}X_u(r)\le -(z_t+t^{\gamma}))\\	&
		\geq \mathbb{P}_{0}(M_s>-t^{\gamma}) -\mathbb{P}_{0}(\max_{r\le s}M_r^{\rho}\ge t^{\gamma}),
	\end{align}
	where the last inequality follows from $M_s^{-\rho}\geq M_s$ and $z_t\geq 0$.
	Combining this with Lemma \ref{lemma_lower_bound_u} and
		 \eqref{lower_bd_M>t^{beta}}, we get
	\begin{align}
		Q_{z_t}(z_t+t^{\gamma},s)	\ge g(s)	-e^{-\alpha s}\mathbf{P}_0(B_s\le -t^{\gamma})
		-e^{-\sqrt{2\alpha}t^{\gamma}}.
	\end{align}
	Letting $N\to \infty$ in \eqref{C_1}  and combining the resulting conclusion with the above, we get
	\begin{align}
				\limsup_{N\to\infty}\limsup_{t\to\infty}
		\frac{C_1(t)}{\mathbf{P}_x^{-\rho}(\tau_0>t,B_t>z_t)}
		\le e^{-\int_{0}^{\infty} \varphi(g(s)) \mathrm{d}s}	=C_{sub}.
	\end{align}
	Applying Lemma  \ref{lemma_tau>t_with_drift} (ii), we get that for $z_t=0$,
	\begin{align}\label{asym_C1-2}
			\limsup_{N\to\infty}\limsup_{t\to\infty}
		C_1(t)\le  C_{sub}(1-e^{2\rho x}),
	\end{align}
	and for $z_t= \sqrt{t}z-\rho t$,
	\begin{align}\label{asym_C1}
				\limsup_{N\to\infty}\limsup_{t\to\infty}
		C_1(t)	\le  \frac{C_{sub}(1-e^{2\rho x})}{\sqrt{2\pi}}\int_{z}^{\infty}e^{-\frac{y^2}{2}} \mathrm{d} y.
	\end{align}
	Next, we show that $\lim_{t\to\infty}C_2(t)=0$. For	$\delta >0$, we have
	\begin{align}
		C_2(t)&\le\mathbf{P}_x^{-\rho}(\tau_0>t,B_t>z_t,  \inf_{r\in [t-N,t]}B_r< z_t+t^{\gamma})\\
		&\le	\mathbf{P}_x^{-\rho}(z_t<B_t<	z_t +\sqrt{t}\delta )\\&\quad+\mathbf{P}_x^{-\rho}(B_t\ge
		z_t+\sqrt{t}\delta ,	\inf_{r\in [t-N,t]}B_r< z_t+t^\gamma ).
	\end{align}
		Since the density of $B_t$ under $\mathbf{P}_x^{-\rho}$ is equal to $\frac{1}{\sqrt{2\pi t}} e^{-\frac{(y-x+\rho t)^2}{2t}} \leq \frac{1}{\sqrt{2\pi t}}$, we have
	\begin{align}
		&		\mathbf{P}_x^{-\rho}(z_t<B_t<	z_t +\sqrt{t}\delta )  \leq \int_{z_t}^{z_t+\sqrt{t}\delta}\frac{1}{\sqrt{2\pi t}}\mathrm{d}y =\frac{\delta}{\sqrt{2\pi}}.
	\end{align}
	Moreover, for any fixed $N>0$, similar to \eqref{B2_prob_upper_bd}, we have for $t\ge N$,
	\begin{align}
		\mathbf{P}_x^{-\rho}(B_t\ge	z_t+\sqrt{t}\delta , \inf_{r\in [t-N,t]}B_r<z_t+t^\gamma) \le
		\mathbf{P}_0\left( |B_N| >	\delta	\sqrt{t}-t^\gamma -N\rho\right).
	\end{align}
	Letting $t\to\infty$ first and then $\delta \to 0$, we get that, for any $\rho< 0$, $\lim_{t\to\infty}C_2(t)=0$.	Combining this with \eqref{upper_bd_Q_rho<0}, \eqref{asym_C1-2} and  \eqref{asym_C1}, we get the desired assertion.
\end{proof}

Now we consider the asymptotic behavior of $Q_z(x,t)$ as $t\to \infty$ for $\rho>0$.  Fix an $N>0$ and define
\begin{align}\label{def_fN}
	f_N^z(y):=	 \mathbf{E}_{y}^{-\rho}\left( 1_{\{\tau_0>N, B_N>z\}}	e^{-\int_{0}^{N}\varphi(
		Q_z(B_s,N-s))	\mathrm{d}s}\right),\quad	y>0,	z\ge 0.
\end{align}
Combining with \eqref{expression_u}, we easily see that
\begin{align}\label{F-and-P}
f_N^z(y)= e^{\alpha N} \P_y\left( \widetilde{M}_N>z\right).
\end{align}

\begin{lemma}\label{lemma_asym_rho>0}
Assume that $\rho>0, x>0$ and $z\geq 0$. It holds that
	\begin{align}
		& \lim_{t\to\infty}t^{3/2}e^{\frac{\rho^2}{2}t}\mathbf{E}_x^{-\rho}\left( 1_{\{\tau_0>t,B_t>z\}}
		e^{-\int_{t-N}^{t}\varphi(Q_z(B_s,t-s))	\mathrm{d}s}\right)	\nonumber\\
		& =	\sqrt{\frac{2}{\pi}} x e^{\rho x}  e^{(\alpha+\frac{\rho^2}{2})N}\int_{0}^{\infty}\P_y\left( \widetilde{M}_N>z\right) ye^{-\rho y}\mathrm{d}y.
	\end{align}
\end{lemma}
\begin{proof}
By the Markov property,
\begin{align}
		 &\mathbf{E}_x^{-\rho}\left(1_{\{\tau_0>t, B_t>z\}}e^{-\int_{t-N}^{t} \varphi(Q_z(B_s,t-s))\mathrm{d}s}\right)	\nonumber\\
	 & =\mathbf{E}_x^{-\rho}\left( 1_{\{\tau_0>t-N\}}	\mathbf{E}_{B_{t-N}}^{-\rho}\left( 1_{\{\tau_0>N, B_N>z\}}e^{-\int_{0}^{N} \varphi(Q_z(B_s,N-s))\mathrm{d}s}\right)\right)\\		&=\mathbf{E}_x^{-\rho}\left( 1_{\{\tau_0>t-N\}}		f_N^z(B_{t-N})\right)		=\mathbf{E}_x^{-\rho}\left( 		f_N^z(B_{t-N})|\tau_0>t-N\right)		\mathbf{P}_x^{-\rho}(\tau_0>t-N).
\end{align}
Applying Lemma \ref{lemma_tau>t_with_drift} (iii), we get that
\begin{align}
	& \lim_{t\to\infty} t^{3/2} e^{\frac{\rho^2}{2}(t-N)} \mathbf{E}_x^{-\rho}\left( f_N^z(B_{t-N})| \tau_0>t-N\right)	\mathbf{P}_x^{-\rho}(\tau_0>t-N)\nonumber\\
	& = \rho^2 \int_{0}^{\infty}f_N^z(y)ye^{-\rho y}\mathrm{d}y \times  \sqrt{\frac{2}{\pi}}	x \rho^{-2}e^{\rho x} =  \sqrt{\frac{2}{\pi}} xe^{\rho x} \int_{0}^{\infty}f_N^z(y)ye^{-\rho y}\mathrm{d}y,
\end{align}
which implies the desired result together with \eqref{F-and-P}.
\end{proof}

{\bf Proofs of Theorem \ref{thm1} and Theorem \ref{thm3}.}
Parts (i) and (ii) of both
Theorem \ref{thm1} and Theorem \ref{thm3} follow
directly from
Lemmas  \ref{lemma_upper_bound_u}, \ref{upper_bound_u}, \ref{lemma3.4} and  \ref{lemma-expectation-survival-number}.
So we only need to prove part (iii) of both theorems. By \eqref{Identity},
it suffices
to prove (iii) of Theorem \ref{thm3}.
Fix $\rho>0$, $N>0$ and $z\geq 0$.
By \eqref{expression_u},  we have for $t\ge N$,
	\begin{align}\label{upper_u_rho>0}
		Q_z(x,t)	&=e^{-\alpha t}		\mathbf{E}_x^{-\rho}\left( 1_{\{\tau_0>t, B_t>z\}}e^{-\int_{0}^{t}\varphi(Q_z(B_s,t-s))\mathrm{d}s}\right)
		\\	&\le e^{-\alpha t}\mathbf{E}_x^{-\rho}\left( 1_{\{\tau_0>t, B_t>z\}}e^{-\int_{t-N}^{t}\varphi(Q_z(B_s,t-s))\mathrm{d}s}\right).
	\end{align}
Applying Lemma \ref{lemma_asym_rho>0}, we get
\begin{align}\label{upper_inf}
	\limsup_{t\to \infty}	t^{3/2}e^{(\alpha+\frac{\rho^2}{2})t}Q_z(x,t)\le \sqrt{\frac{2}{\pi}}
	x e^{\rho x}  e^{(\alpha+\frac{\rho^2}{2})N}\int_{0}^{\infty}\P_y\left( \widetilde{M}_N>z\right) ye^{-\rho y}\mathrm{d}y.
\end{align}
It follows from \eqref{upper_bound_u-2} that
	\begin{align}\label{lower_u_rho>0}
		Q_z(x,t)\ge e^{-\alpha t}\mathbf{E}_x^{-\rho}\left(
		1_{\{\tau_0>t, B_t>z\}}e^{-\int_{t-N}^{t}\varphi(Q_z(B_s,t-s))\mathrm{d}s}\right)
		e^{-\int_{0}^{t-N}\varphi(g(t-s))\mathrm{d}s}.
	\end{align}
Recall that the moment condition \eqref{LLogL-moment-condition} is equivalent to \eqref{Integral-of-phi-is-finite}, which implies that
\begin{align}\label{integer_left_part}
	1\ge e^{-\int_{0}^{t-N}\varphi(g(t-s))\mathrm{d}s}	=e^{-\int_{N}^{t}\varphi(g(s))\mathrm{d}s}
	\ge e^{-\int_{N}^{\infty}\varphi(g(s))\mathrm{d}s} \stackrel{N\to\infty}{\longrightarrow} 1.
\end{align}
Using Lemma \ref{lemma_asym_rho>0} again, we get
\begin{align}\label{lower_sup}
	&\liminf_{t\to \infty}	t^{3/2}e^{(\alpha+\frac{\rho^2}{2})t}	Q_z(x,t)\\
	&\ge	e^{-\int_{N}^{\infty}\varphi(g(s))\mathrm{d}s}	\sqrt{\frac{2}{\pi}}
	x e^{\rho x}    e^{(\alpha+\frac{\rho^2}{2})N}\int_{0}^{\infty}\P_y\left( \widetilde{M}_N>z\right) ye^{-\rho y}\mathrm{d}y.
\end{align}
Letting $N\to \infty$,
combining
 \eqref{upper_inf} and \eqref{lower_sup}, we get
\begin{align}\label{limit_u}
	&\lim_{t\to \infty}	t^{3/2}e^{(\alpha+\frac{\rho^2}{2})t}	Q_z(x,t)	\\ &= \sqrt{\frac{2}{\pi}}
	x e^{\rho x}  \lim_{N\to \infty}e^{(\alpha+\frac{\rho^2}{2})N}\int_{0}^{\infty}\P_y\left( \widetilde{M}_N>z\right) ye^{-\rho y}\mathrm{d}y= \sqrt{\frac{2}{\pi}}	x e^{\rho x}
	C_z(\rho),
\end{align}
where  $ C_z(\rho):=\lim_{N\to \infty}e^{(\alpha+\frac{\rho^2}{2})N}\int_{0}^{\infty}\P_y\left( \widetilde{M}_N>z\right) ye^{-\rho y}\mathrm{d}y.$ Now we show that $C_z(\rho)\in (0,\infty)$.
First, applying \eqref{upper_u_rho>0} with $N=t$, we get
\begin{align}
Q_z(x,t)	\le e^{-\alpha t}\mathbf{P}_x^{-\rho}\left(\tau_0>t, B_t>z\right).
\end{align}
Combining this with Lemma \ref{lemma_tau>t_with_drift} we get that
\begin{align}\label{upper_bd_limsup_u}
	\limsup_{t\to\infty}t^{\frac{3}{2}}e^{(\alpha +\frac{\rho^2}{2})t}
	Q_z(x,t)	\le \sqrt{\frac{2}{\pi}} x e^{\rho x} \int_{z}^{\infty}ye^{-\rho y} \mathrm{d}y.
\end{align}
Therefore,
 $C_z(\rho)\leq  \int_{z}^{\infty}ye^{-\rho y} \mathrm{d}y<\infty$. Next, by  \eqref{upper_bound_u-2}, we have
\begin{align}\label{lower-u-x-t}
	& Q_z(x,t)	\ge e^{-\int_{0}^{t}\varphi(g(s))\mathrm{d}s}e^{-\alpha t}
		\mathbf{P}_x^{-\rho}\left(\tau_0>, B_t>z\right)	\\	& \ge C_{sub} e^{-\alpha t}			\mathbf{P}_x^{-\rho}\left(\tau_0>, B_t>z\right),
\end{align}
where the last inequality follows from \eqref{Constant-C-sub}.
Using Lemma
\ref{lemma_tau>t_with_drift} (iii)
again, we get
\begin{align}
	\liminf_{t\to\infty}t^{\frac{3}{2}}e^{(\alpha +\frac{\rho^2}{2})t}  Q_z(x,t)
	\ge C_{sub}	\sqrt{\frac{2}{\pi}} x e^{\rho x} \int_{z}^{\infty}ye^{-\rho y} \mathrm{d}y.
\end{align}
Therefore, we see that $C_z(\rho)\geq C_{sub} \int_{z}^{\infty}ye^{-\rho y} \mathrm{d}y>0$.
Combining
\eqref{limit_u}  and  Lemma \ref{lemma-expectation-survival-number}, we get 
%the desired assertion for $\rho>0$.
\eqref{e:thrm1.3iiib}.
\qed

\textbf{Proof of Corollary \ref{cor1}:} We only give the proof of (iii). Taking $N=t$ in
\eqref{upper_u_rho>0},
by \eqref{lower-u-x-t} with $z=0$, we have
\begin{align}
	& \P_x\left(\widetilde{M}_t^{-\rho}>z \big| \widetilde{\zeta}^{-\rho}>t \right) =\frac{Q_z(x,t)}{u(x,t)} \nonumber\\
	&\leq  \frac{\mathbf{P}_x^{-\rho}\left(\tau_0>t, B_t>z\right) }{C_{sub} \mathbf{P}_x^{-\rho}\left(\tau_0>t\right)}= \frac{1}{C_{sub}}\mathbf{P}_x^{-\rho}\left(B_t>z\big| \tau_0>t\right).
\end{align}
By Lemma \ref{lemma_tau>t_with_drift} (iii), the tightness of $\widetilde{M}_t^{-\rho}$ follows
from
 the tightness of $B_t$ under $\mathbf{P}_x^{-\rho}\left(\cdot \big| \tau_0>t\right)$. Therefore, the weak convergence of $\widetilde{M}_t^{-\rho}$ is a consequence of  the existence of $C_z(\rho)$ in Theorem \ref{thm3} (iii), which implies the desired result.
\qed

\section{Proof of Theorem \ref{thm2}}\label{proof of thm2}

{\bf Proof of Theorem \ref{thm2}.}
Recall that $v(x, y)=\P_x(\widetilde{M}^{-\rho}>y)$, $x,y>0$.
We divide the proof into three steps. In Step 1, we use the Feynman-Kac formula and the strong Markov property to rewrite $v(x,y)$ as the product of two factors $A_1(x,y)$ and  $A_2(y)$, see  \eqref{second_expression_v} below. In Steps 2 and 3, we study the asymptotic behavior of $A_1(x,y)$ and
$A_2(y)$ as $y\to\infty$ respectively. Combining these
results, we arrive at the assertion of the theorem.

{\bf Step 1:}
For $0< x<y$, comparing the first branching time with $\tau_y$, we have
\begin{align}
	&v(x,y)=\int_{0}^{\infty} \beta e^{-\beta s} \mathbf{P}_x^{-\rho}(\tau_y< \tau_0,\tau_y\le s)\mathrm{d}s\\
	&+\int_{0}^{\infty}\beta e^{-\beta s}\mathbf{E}_x^{-\rho}\left(\left(1-\sum_{k=0}^{\infty} p_k\left(1-v(B_s,y)\right)^k \right)  1_{\{\tau_y\land \tau_0>s\}}\right)\\
	&=\mathbf{E}_x^{-\rho}	\left(e^{-\beta \tau_y}1_{\{\tau_y<\tau_0\}}\right)
	+\int_{0}^{\infty}\beta e^{-\beta s}\mathbf{E}_x^{-\rho}\left( \left(1-\sum_{k=0}^{\infty}p_k \left(1-v(B_s,y)\right)^k\right)1_{\{\tau_y\land \tau_0>s\}}\right)\mathrm{d}s.
\end{align}
By \cite[Lemma 4.1]{Dynkin01}, the above equation is equivalent to
\begin{align}
	&v(x,y)+\beta \int_{0}^{\infty}\mathbf{E}_x^{-\rho}\left(v(B_s,y)1_{\{\tau_y\land \tau_0>s\}}\right)\mathrm{d}s\\	&=\mathbf{P}_x^{-\rho}\left(\tau_y<\tau_0\right)
	+\beta\int_{0}^{\infty}\mathbf{E}_x^{-\rho}\left( \left(1-\sum_{k=0}^{\infty}p_k\left(1-v(B_s,y)\right)^k\right)1_{\{\tau_y\land \tau_0>s\}}\right)\mathrm{d}s,
\end{align}
which is also equivalent to
\begin{align}
	v(x,y)=\mathbf{P}_x^{-\rho}\left(\tau_y<\tau_0\right)
	- \mathbf{E}_x^{-\rho}\left( \int_{0}^{\tau_y\land\tau_0} \Phi(v(B_s,y))\mathrm{d}s\right),
\end{align}
where $\Phi$ is defined in \eqref{def_Phi}.
Using the Feynman-Kac formula, we get that
	\begin{align}\label{expression_v}
		v(x,y)	&=\mathbf{E}_x^{-\rho}\left( 1_{\{\tau_y<\tau_0\}} e^{-\alpha\tau_y-\int_{0}^{\tau_y}\varphi(v(B_s,y)) \mathrm{d}s}\right)
			\\	&=\frac{x}{y}e^{\rho(x-y)}\mathbf{E}_x^B\left( e^{-(\alpha+\frac{\rho^2}{2}) \tau_y-\int_{0}^{\tau_y}\varphi(v(B_s,y))\mathrm{d}s}\right),\nonumber
	\end{align}
where the last equality follows from Lemma \ref{lemma_Bessel_driftBM}.
Combining the second inequality in \eqref{expression_v}  and \eqref{expression_BM_driftBM} (with $h=0$),
it holds that
\begin{align}\label{upper-bound-of-v}
		& v(x,y)	\leq \mathbf{E}_x^{-\rho}\left(e^{-\alpha\tau_y}\right) 		=e^{\rho (x-y)} \mathbf{E}_x\left( e^{-(\alpha+\frac{\rho^2}{2})\tau_y}\right)	= e^{\left(\rho+\sqrt{2\alpha+\rho^2} \right)(x-y)}.
\end{align}
Fix a $\gamma\in (0,1)$. By the strong Markov property of Bessel-3 processes, we have
\begin{align}\label{second_expression_v}
	& v(x,y)=\frac{x}{y}	e^{\rho (x-y)}	\mathbf{E}_x^B\left( e^{-(\alpha+\frac{\rho^2}{2})
	\tau_{(y-y^{\gamma})}	-\int_{0}^{\tau_{(y-y^{\gamma})}}	\varphi(v(B_s,y))\mathrm{d}s}\right)
	\\	& \quad\quad\times \mathbf{E}_{y-y^{\gamma}}^B\left( e^{-(\alpha+\frac{\rho^2}{2})\tau_{y} -\int_{0}^{\tau_{y}}\varphi(v(B_s,y))\mathrm{d}s}\right)\nonumber\\
	&=: \frac{x}{y}	e^{\rho (x-y)}	A_1(x,y)	A_2(y),\nonumber
\end{align}
where
\begin{align}
	A_1(x,y):=\mathbf{E}_x^B\left(e^{-(\alpha+\frac{\rho^2}{2})	\tau_{(y-y^{\gamma})}	-\int_{0}^	{\tau_{(y-y^{\gamma})}}	\varphi(v(B_s,y))\mathrm{d}s}\right)
\end{align}
and
\begin{align}
	A_2(y)	:=\mathbf{E}_{y-y^{\gamma}}^B\left( e^{-(\alpha+\frac{\rho^2}{2})\tau_{y} -\int_{0}^{\tau_{y}}\varphi(v(B_s,y))\mathrm{d}s}\right).
\end{align}

{\bf Step 2:}
In this step, we study the asymptotic behavior of
$A_1(x, y)$ as $y\to\infty$.
By Lemma  \ref{lemma_Bessel_driftBM} with 
%$\rho$ replaced by $-\sqrt{2\alpha +\rho^2}$, $a=0, y=y-y^\gamma$ and $h=\varphi\circ v(\cdot, y)$,
$a=0$, $\rho$ replaced by $-\sqrt{2\alpha +\rho^2}$, $y$ replaced by $y-y^\gamma$, and $h=\varphi\circ v(\cdot, y)$,
 we get
\begin{align}
	&A_1(x,y)= 
	\frac{y-y^\gamma}{x}
	e^{-\sqrt{2\alpha+\rho^2}(y-y^\gamma-x)} \mathbf{E}_x^{\sqrt{2\alpha+\rho^2}} \left(1_{\{\tau_{(y-y^\gamma)} <\tau_0 \}}  e^{-\int_0^{\tau_{(y-y^\gamma)}} \varphi(v(B_s,y))\mathrm{d}s}\right)\nonumber\\
	&=: 
		\frac{y-y^\gamma}{x}
	e^{-\sqrt{2\alpha+\rho^2}(y-y^\gamma-x)} \hat{A}_1(x,y).
\end{align}
By the inequality $1-e^{-|x|}\leq |x|$,
we obtain that
\begin{align}\label{Diffenrence}
   & 0\leq  \mathbf{P}_x^{\sqrt{2\alpha+\rho^2}} \left( \tau_{(y-y^\gamma)} <\tau_0\right) - \hat{A}_1(x,y)
   \\   &= \mathbf{E}_x^{\sqrt{2\alpha+\rho^2}} \left(1_{\{\tau_{(y-y^\gamma)} <\tau_0 \}} \left(1-  e^{-\int_0^{\tau_{(y-y^\gamma)}} \varphi(v(B_s,y))\mathrm{d}s}\right)\right)\nonumber\\
   & \leq \mathbf{E}_x^{\sqrt{2\alpha+\rho^2}} \left(\int_0^{\tau_{(y-y^\gamma)}} \varphi(v(B_s,y))\mathrm{d}s\right).\nonumber
\end{align}
%Now set $y_*(x):= \inf\{t\geq y-y^\gamma: t- x\in \mathbb{N}\}$
Now set $y_*(x):= \inf\{w\geq y-y^\gamma: w- x\in \mathbb{N}\}$ to be the smallest number $w$ greater than or equal to $y-y^\gamma$ such that $w-x$ is a positive integer
and
$c_*:= \rho +\sqrt{2\alpha+\rho^2}>0$.
 By \eqref{upper-bound-of-v},
\begin{align}\label{Upper-Integral-of-v} 	
	& \mathbf{E}_x^{\sqrt{2\alpha+\rho^2}} \left(\int_0^{\tau_{(y-y^\gamma)}} \varphi(v(B_s,y))\mathrm{d}s\right)\leq \mathbf{E}_x^{\sqrt{2\alpha+\rho^2}} \left(\int_0^{\tau_{y_*(x)}} \varphi(e^{c_*(B_s-y)})\mathrm{d}s\right)\\
	& =
	\sum_{k=0}^{y_*(x)-x-1} 
	\mathbf{E}_x^{\sqrt{2\alpha+\rho^2}} \left(\int_{\tau_{x+k}}^{\tau_{x+k+1}} \varphi(e^{c_*(B_s-y)})\mathrm{d}s\right) \nonumber\\
	&  \leq  
	\sum_{k=0}^{y_*(x)-x-1} 
	 \mathbf{E}_x^{\sqrt{2\alpha+\rho^2}} \left(\tau_{x+k+1}- \tau_{x+k} \right)\varphi(e^{c_*(x+k+1-y)}) \nonumber\\
	&= \mathbf{E}_0^{\sqrt{2\alpha+\rho^2}} \left(\tau_{1}\right) 
	 \sum_{k=1}^{y_*(x)-x} 
	\varphi\left(
  e^{-c_*(y-1-y_*(x)+k)}\right).\nonumber
\end{align}
%According to the definition of $y_*(x)$,
According to the definition of $y_*(x)$, for $y$ large enough, 
\[
y-1-y_*(x) \geq y-1-(y-y^\gamma+1)= y^\gamma-2.
\]
Therefore, when $y$ is large enough so that 
$y^\gamma -2 \geq  y^{\gamma/2}$,
by Lemma \ref{upper-k}, we have
\begin{align}\label{Upper-Difference}
	& \mathbf{E}_x^{\sqrt{2\alpha+\rho^2}} \left(\int_0^{\tau_{(y-y^\gamma)}} \varphi(v(B_s,y))\mathrm{d}s\right)	\\	& \leq \mathbf{E}_0^{\sqrt{2\alpha+\rho^2}} \left(\tau_{1}\right) \sum_{k=1}^{\infty } \varphi\left(	e^{-c_*(y^{\gamma/2}+k)}\right)  \leq	\mathbf{E}_0^{\sqrt{2\alpha+\rho^2}} \left(\tau_{1}\right)
	\int_0^\infty \varphi\left(	e^{-c_*(y^{\gamma/2}+z)}\right) \mathrm{d}z\nonumber\\
	& =	\mathbf{E}_0^{\sqrt{2\alpha+\rho^2}} \left(\tau_{1}\right)		 \int_{y^{\gamma/2}}^\infty \varphi\left( e^{-c_*z}\right) \mathrm{d}z		 \stackrel{y\to\infty}{\longrightarrow} 0.\nonumber
\end{align}
Combining the above limit with  \eqref{Diffenrence},
it holds that
\begin{align}
    \lim_{y\to\infty} \left(\mathbf{P}_x^{\sqrt{2\alpha+\rho^2}} \left( \tau_{(y-y^\gamma)} <\tau_0\right) - \hat{A}_1(x,y)\right)=0.
\end{align}
Combining
 \eqref{Scale-function} and the definition of $\hat{A}_1$, we conclude that
\begin{align}\label{asymptotic_A1}
	& \lim_{y\to\infty} \frac{A_1(x,y)}{y}e^{\sqrt{2\alpha+\rho^2}(y-y^\gamma)}=\frac{e^{\sqrt{2\alpha+\rho^2}x}}{x} \lim_{y\to\infty}  \mathbf{P}_x^{\sqrt{2\alpha+\rho^2}} \left( \tau_{(y-y^\gamma)} <\tau_0\right)	\\	& = 	\frac{2}{x}	 \sinh\left(x\sqrt{2\alpha+\rho^2}\right).\nonumber
\end{align}

{\bf Step 3:}
In this step, we study the limit behavior for $A_2$.
By Lemma \ref{lemma_Bessel_driftBM}, we have
	\begin{align} \label{limit-A-2}
		&\qquad A_2(y)	=\frac{y e^{-\sqrt{2\alpha+\rho^2} y^{\gamma}}}{y-y^{\gamma}}\mathbf{E}_{y-y^{\gamma}}^{\sqrt{2\alpha+\rho^2}}\left( 1_{\{\tau_{y}<\tau_0\}}e^{-\int_{0}^{\tau_{y}}\varphi(v(B_s,y)) \mathrm{d}s}\right)			\\
		&=\frac{y e^{-\sqrt{2\alpha+\rho^2}y^{\gamma}}	}{y-y^{\gamma}}
		\left(\mathbf{E}_{y-y^{\gamma}}^{\sqrt{2\alpha+\rho^2}}\left( e^{-\int_{0}^{\tau_{y}}\varphi(v(B_s,y))\mathrm{d}s}\right)-\mathbf{E}_{y-y^{\gamma}}^{\sqrt{2\alpha+\rho^2}}\left(1_{\{\tau_{y}\ge \tau_0\}}e^{-\int_{0}^{\tau_{y}}\varphi(v(B_s,y))\mathrm{d}s}\right)\right),\nonumber
\end{align}
where, under $\mathbf{P}_{y-y^{\gamma}}^{\sqrt{2\alpha+\rho^2}}$, $B$ is a Brownian motion with drift $\sqrt{2\alpha+\rho^2}$ starting from $y-y^{\gamma}$. We claim that
\begin{align}
	&\lim_{y\to\infty}\mathbf{E}_{y-y^{\gamma}}^{\sqrt{2\alpha+\rho^2}}\left( e^{-\int_{0}^{\tau_{y}}\varphi(v(B_s,y))\mathrm{d}s}\right) =	C_*(\rho)   \in (0,\infty),\label{Claim1}\\	& \lim_{y\to\infty} \mathbf{E}_{y-y^{\gamma}}^{\sqrt{2\alpha+\rho^2}}\left( 1_{\{\tau_{y}\ge \tau_0\}}e^{-\int_{0}^{\tau_{y}}\varphi(v(B_s,y))\mathrm{d}s}\right)=0. \label{Claim2}
\end{align}
We prove \eqref{Claim2} first. In fact,
 by Lemma \ref{lemma_Bessel_driftBM} and \ref{lemma_property_Bessel3}, we have
\begin{align}
	& \mathbf{E}_{y-y^{\gamma}}^{\sqrt{2\alpha+\rho^2}}\left( 1_{\{\tau_{y}\ge \tau_0\}}e^{-\int_{0}^{\tau_{y}}\varphi(v(B_s,y))\mathrm{d}s}\right)
	\le \mathbf{P}_{y-y^{\gamma}}^{\sqrt{2\alpha+\rho^2}}(\tau_{y}\ge \tau_0)
	=1-\mathbf{P}_{y-y^{\gamma}}^{\sqrt{2\alpha+\rho^2}}(\tau_{y}< \tau_0)\\
	&=1-\frac{y-y^{\gamma}}{y}e^{-\sqrt{2\alpha+\rho^2}y^{\gamma}} \mathbf{E}_{y-y^{\gamma}}^B\left( e^{-\frac{2\alpha+\rho^2}{2}\tau_y}\right)\\
	&	=1-e^{-\sqrt{2\alpha+\rho^2}y^{\gamma}} \frac{\sinh((y-y^{\gamma})\sqrt{2\alpha+\rho^2})}{\sinh(y\sqrt{2\alpha+\rho^2})}
	\stackrel{y\to\infty}{\longrightarrow}0,
\end{align}
which gives \eqref{Claim2}. To prove \eqref{Claim1},  for any $y>0$,
define
\begin{align}
G(y):=\mathbf{E}_{y-y^{\gamma}}^{\sqrt{2\alpha+\rho^2}}\left( e^{-\int_{0}^{\tau_{y}}\varphi(v(B_s,y))\mathrm{d}s}\right).
\end{align}
For $z>y$, by the strong Markov property, we have
\begin{align*}
		&G(z)	=\mathbf{E}_{0}^{\sqrt{2\alpha+\rho^2}}\left( e^{-\int_{0}^{\tau_{z^{\gamma}}} \varphi(v(B_s+z-z^{\gamma},z))\mathrm{d}s}\right)	\\	&=\mathbf{E}_{0}^{\sqrt{2\alpha+\rho^2}} \left( e^{-\int_{0}^{\tau_{(z^{\gamma}-y^{\gamma})}}\varphi(v(B_s+z-z^{\gamma},z))\mathrm{d}s}\right)	\mathbf{E}_{z^{\gamma}-y^{\gamma}}^{\sqrt{2\alpha+\rho^2}}\left(  e^{-\int_{0}^{\tau_{z^{\gamma}}}\varphi(v(B_s+z-z^{\gamma},z))\mathrm{d}s}\right)\nonumber
\end{align*}
The first term of the above display is dominated by $1$ from above, and the second  is
equal to
 $\mathbf{E}_{0}^{\sqrt{2\alpha+\rho^2}}\left( e^{-\int_{0}^{\tau_{y^{\gamma}}}\varphi(v(B_s+z-y^{\gamma},z))\mathrm{d}s}\right)$.
 Hence,  $G(z)$ is bounded from above by
\begin{equation}\label{G}
G(z)\leq \mathbf{E}_{0}^{\sqrt{2\alpha+\rho^2}}\left( e^{-\int_{0}^{\tau_{y^{\gamma}}} \varphi(v(B_s+y-y^{\gamma}+z-y,y+z-y))\mathrm{d}s}\right).
\end{equation}
Note that, for $w>0$, it holds that
\begin{align}
	& v(x+w,y+w)	=\mathbb{P}_{x+w}(\exists \ t>0,\ u\in N_t^{-\rho}\ s.t. \min_{s\leq t} X_u(s)>0,\ X_u(t)>y+w ) \nonumber\\ &	\ge \mathbb{P}_{x+w}(\exists \ t>0,\ u\in N_t^{-\rho}\ s.t. \min_{s\leq t} X_u(s)>w,\ X_u(t)>y+w )=v(x,y).
\end{align}
Combining this with \eqref{G} we get that
\begin{align}
	G(z)\le\mathbf{E}_{0}^{\sqrt{2\alpha_+\rho^2}}\left( e^{-\int_{0}^{\tau_{y^{\gamma}}}\varphi(v(B_s+y-y^{\gamma},y))\mathrm{d}s}\right)	=G(y),\quad	z>y.
\end{align}
Thus the limit $C_*(\rho):=\lim_{y\to \infty}G(y)$ exists. Combining \eqref{limit-A-2}, \eqref{Claim1} and \eqref{Claim2}, we get
\begin{align}\label{asym-A-2}
	\lim_{y\to \infty}	A_2(y) e^{\sqrt{2\alpha+\rho^2}y^{\gamma}}=	C_*(\rho).
\end{align}
Now we show that $C_*(\rho)$ is finite and positive. The finiteness follows trivially from \eqref{Claim1}. To show $C_*(\rho)>0$, we assume without loss of generality that $y$ is an integer. By the strong Markov property and Jensen's inequality,
\begin{align}
	&\mathbf{E}_{y-y^{\gamma}}^{\sqrt{2\alpha+\rho^2}}\left( e^{-\int_{0}^{\tau_{y}}\varphi(v(B_s,y))\mathrm{d}s}\right) = \frac{\mathbf{E}_{0}^{\sqrt{2\alpha+\rho^2}}\left(e^{-\int_{0}^{\tau_{y}}\varphi(v(B_s,y))\mathrm{d}s}\right)}{\mathbf{E}_{0}^{\sqrt{2\alpha+\rho^2}}\left( e^{-\int_{0}^{\tau_{(y-y^\gamma)}}\varphi(v(B_s,y))\mathrm{d}s}\right) }\nonumber\\
	& \geq \mathbf{E}_{0}^{\sqrt{2\alpha+\rho^2}}\left( e^{-\int_{0}^{\tau_{y}}\varphi(v(B_s,y))\mathrm{d}s}\right) \geq \exp\left\{- \sum_{n=1}^y \mathbf{E}_0^{\sqrt{2\alpha+\rho^2}} \int_{\tau_{n-1}}^{\tau_n} \varphi(v(B_s,y)) \mathrm{d}s \right\}.
\end{align}
For $\tau_{n-1} \leq s\leq \tau_n$, by Lemma \ref{upper-k} and \eqref{upper-bound-of-v},
\[
\int_{\tau_{n-1}}^{\tau_n} \varphi(v(B_s,y)) \mathrm{d}s \leq (\tau_n -\tau_{n-1}) \varphi(v(n,y))\leq (\tau_n -\tau_{n-1})\varphi\left( e^{(n-y)(\sqrt{2\alpha+\rho^2}+\rho)}   \right).
\]
Note that, under $\mathbf{P}_0^{\sqrt{2\alpha+\rho^2}}$,
$\{\tau_{n}- \tau_{n-1}\}_{n\geq 1}$ are iid random variables with finite first moment. Therefore,
\begin{align}
	&\mathbf{E}_{y-y^{\gamma}}^{\sqrt{2\alpha+\rho^2}}\left( e^{-\int_{0}^{\tau_{y}}\varphi(v(B_s,y))\mathrm{d}s}\right)	\geq \exp\left\{-\mathbf{E}_0^{\sqrt{2\alpha+\rho^2}} \left(\tau_1\right)  \sum_{n=1}^y
	\varphi\left( e^{(n-y)(\sqrt{2\alpha+\rho^2}+\rho)}   \right)
	  \right\}\nonumber\\
	& = \exp\left\{-\mathbf{E}_0^{\sqrt{2\alpha+\rho^2}} \left(\tau_1\right)  \sum_{n=0}^{y-1}
	\varphi\left( e^{-n(\sqrt{2\alpha+\rho^2}+\rho)}   \right)	\right\}\nonumber\\
	&\geq \exp\left\{-\mathbf{E}_0^{\sqrt{2\alpha+\rho^2}} \left(\tau_1\right)  \sum_{n=0}^{\infty}
	\varphi\left( e^{-n(\sqrt{2\alpha+\rho^2}+\rho)}   \right) \right\},
\end{align}
which implies that
\begin{align}
	C_*(\rho)	\geq \exp\left\{-\mathbf{E}_0^{\sqrt{2\alpha+\rho^2}} \left(\tau_1\right)  \sum_{n=0}^{\infty}
	\varphi\left( e^{-n(\sqrt{2\alpha+\rho^2}+\rho)}   \right) \right\}.
\end{align}
Now by Lemma \ref{upper-k}, we have
\begin{align}
	\sum_{n=0}^{\infty} 	\varphi\left( e^{-n(\sqrt{2\alpha+\rho^2}+\rho)}   \right) \leq \varphi(1)+\int_0^\infty 	\varphi\left( e^{-z(\sqrt{2\alpha+\rho^2}+\rho)}   \right) \mathrm{d}z<\infty,
\end{align}
which implies that $C_*(\rho)>0$. Combining \eqref{second_expression_v},  \eqref{asymptotic_A1} and \eqref{asym-A-2}, we conclude that
\begin{align}\label{asymptotic_v}
	\lim_{y\to\infty}	e^{(\sqrt{2\alpha+\rho^2}+\rho) y} v(x,y)= 2C_*(\rho)	e^{\rho x}	\sinh(x\sqrt{2\alpha+\rho^2}),
\end{align}
which completes the proof of the theorem.
\qed

%	\vspace{.1in}
%	\textbf{Acknowledgment}:

\end{document}